\documentclass[preprint,11pt]{article}
\usepackage{fullpage}
\ifdefined\usebigfont

\usepackage{boondox-cal}
\usepackage{times}
\usepackage[italicdiff]{physics}
\usepackage[fontsize=13pt]{scrextend}
\usepackage[left=1.56in,right=1.56in,top=1.74in,bottom=1.74in]{geometry}
\usepackage{bigints}
\usepackage{csquotes}
\usepackage{tikz}
\else
\fi

\usepackage{amssymb,amsfonts,amsmath,amsthm,amscd,dsfont,mathrsfs,bbm}
\usepackage{graphicx,float,psfrag,epsfig,amssymb}
\usepackage[usenames,dvipsnames,svgnames,table]{xcolor}
\definecolor{darkgreen}{rgb}{0.0,0,0.9}
\usepackage[pagebackref,letterpaper=true,colorlinks=true,pdfpagemode=none,citecolor=OliveGreen,linkcolor=BrickRed,urlcolor=BrickRed,pdfstartview=FitH]{hyperref}
\usepackage{wrapfig}
\usepackage{relsize}
\usepackage{color}
\usepackage{pict2e}
\usepackage{subcaption}
\usepackage{nameref}
\usepackage{makecell}
\usepackage[font={small}]{caption} 
\usepackage{mathtools}
\usepackage{enumitem}
\usepackage{extpfeil}

\usepackage{xurl}
\usepackage{textcomp}

\usepackage{pgfplots}
\usepackage{stackengine}
\usepackage{algorithm}
\usepackage{algpseudocode}
\let\chapter\section
\usepackage{diagbox}

\DeclareMathAlphabet{\mathpzc}{OT1}{pzc}{m}{it}

\newtheorem{propo}{Proposition}
\newtheorem{lemma}[propo]{Lemma}
\newtheorem{coro}[propo]{Corollary}
\newtheorem{thm}[propo]{Theorem}

\theoremstyle{definition}
\newtheorem{assumption}{Assumption}
\newtheorem{defi}{Definition}
\newtheorem{rmk}{Remark}
\newtheorem{example}{Example}

\usepackage[authoryear]{natbib}

\usepackage{url}

\addtocontents{toc}{\protect\setcounter{tocdepth}{2}}

\def\tp{{\tilde{p}}}

\def\reals{{\mathbb R}}

\def\eps{{\varepsilon}}
\def\prob{{\mathbb P}}
\def\E{{\mathbb E}}

\def\L0{{L_i}}

\def\<{\langle}
\def\>{\rangle}

\def\F{{\sf F}}
\def\ind{{\mathbb I}}

\def\F{{\sf F}}

\def\sT{{\sf T}}

\def\hz{\widehat{z}}

\def\v*{v_i}
\def\T*{T_i}

\def\u*{u_i}
\def\F*{F_i}

\def\bR{{\mtx{R}}}

\def\l1u{W}

\newcommand{\ajcomment}[1]{}

\makeatletter
\newcommand{\labitem}[2]{%
\def\@itemlabel{\text{#1}}
\item
\def\@currentlabel{#1}\label{#2}}
\makeatother

\DeclareMathAlphabet{\mathpzc}{OT1}{pzc}{m}{it}

\def\cP{\mathcal{P}}

\def\unif{\mathsf{Unif}}

\def \tp{\tilde{p}}

\def \fC{\mathfrak{C}}
\def \conv{\mathsf{conv}}
\def \sH{d_\mathsf{H}}
\def \bR{\mathbf{R}}
\def \hopt{\widehat{\mathsf{R}}}
\def \myom{\mathsf{DM}}
\def \mydp{\mathsf{PM}}

\def \To{\rightrightarrows}

\def\bC{\mathbf{C}}

\newcommand{\RR}{\mathbb{R}}
\newcommand{\p}[1]{\left(#1\right)}

\newcommand{\cb}[1]{\left\{#1\right\}}

\newcommand{\EE}[2][]{\mathbb{E}_{#1}\left[#2\right]}

\newcommand{\argmin}{\operatorname{argmin}}
\newcommand{\argmax}{\operatorname{argmax}}

\newcommand\blfootnote[1]{%
  \begingroup
  \renewcommand\thefootnote{}\footnote{#1}%
  \addtocounter{footnote}{-1}%
  \endgroup
}

\title{Optimal Mechanisms for Demand Response:\\ An Indifference Set Approach}
\author{Mohammad Mehrabi \\ \texttt{mehrabi@stanford.edu}
\and Omer Karaduman\\  \texttt{omerkara@stanford.edu} \and Stefan Wager \\ \texttt{swager@stanford.edu}}
\date{Draft version \today}

\begin{document}

\maketitle


\begin{abstract}
The time at which renewable (e.g., solar or wind) energy resources produce electricity cannot generally
be controlled. In many settings, however, consumers have some flexibility in their energy consumption needs,
and there is growing interest in demand-response programs that leverage this flexibility to shift energy
consumption to better match renewable production---thus enabling more efficient utilization of these resources.
We study optimal demand response in a setting where consumers use home energy management systems (HEMS) to autonomously adjust their electricity consumption. Our core assumption is that HEMS operationalize flexibility by querying the consumer for their preferences and computing the ``indifference set'' of all energy consumption profiles that can be used to satisfy these preferences. Then, given an indifference set, HEMS can respond to grid signals while guaranteeing user-defined comfort and functionality; e.g., if a consumer sets a temperature range, a HEMS can precool and preheat to align with peak renewable production, thus improving efficiency without sacrificing comfort. We show that while price-based mechanisms are not generally optimal for demand response, they become asymptotically optimal in large markets under a mean-field limit. Furthermore, we show that optimal dynamic prices can be efficiently computed in large markets by only querying HEMS about their planned consumption under different price signals. Using an OpenDSS-powered grid simulation for Phoenix, Arizona, we demonstrate that our approach enables meaningful demand response without creating grid instability.
\end{abstract}

\section{Introduction}

Consumer electricity demand is largely unmanaged:\blfootnote{\hspace{-7mm} We are grateful to Andrey Bernstein,  Mahmoud Saleh, Jing Shang and Xinyang Zhou for valuable feedback and numerous helpful conversations; and also thank Xinyang Zhou for helping us get set up with the OpenDSS grid simulator. This research was funded by a grant from Trellis Climate (Prime Coalition) to advance deep demand flexibility solutions and a grant from the Office of Naval Research.}
In most markets today, retail customers obtain electricity via fixed-rate contracts that do not adapt to real-time grid conditions or reflect the true costs of electricity production.\footnote{Many regions
also run wholesale energy markets where electricity is traded at floating prices between
utilities, power producers, municipalities, etc.~\citep{wolak2021wholesale}. End consumers,
however, generally do not trade directly on these markets; and exposing consumers to wholesale
prices is generally considered to lead to unacceptable risks, as illustrated by the Griddy Energy case below.}
This unmanaged demand is in tension with renewable energy resources, such as solar and wind:
Unlike conventional energy resources, renewables yield intermittent production and cannot be
dispatched on demand, thus leading to frequent mismatches between electricity production
and consumption \citep{gowrisankaran2016intermittency,cordera2023unit}.
This supply-demand mismatch then often forces the grid to depend on expensive and less efficient
power plants during peak demand periods, leading to increased operational expenses and
environmental impacts \citep{bird2013integrating, IMPRAM2020100539, michaelides2019fossil}.

As renewable energy sources like solar and wind become more prevalent---for example, in the U.S., electricity demand is expected to grow by approximately 1\% annually from 2022 to 2050, with renewable sources projected to supply 44\% of electricity by 2050, up from 21\% in 2021 \citep{eia2023aeo}---these inefficiencies are projected to intensify further,  mainly due to the mismatch between renewable energy availability and peak demand periods.
This mismatch is further exacerbated by the growing adoption of rooftop solar panels, which allow households to both generate and consume electricity, thereby causing local demand profiles to fluctuate unpredictably. These fluctuations require more sophisticated balancing mechanisms to maintain grid stability \citep{caballero2022distributed}, highlighting the need for more adaptive and responsive demand management solutions.

These challenges have led to increasing interest in demand-response programs that seek to leverage consumer flexibility
to achieve demand patterns better aligned with renewable production \citep{yoon2014demand,godina2018model,siano2014demand}. For example, demand response programs may encourage consumers to engage in pre-cooling: On a hot day, it may be advantageous to apply air conditioning early in the afternoon while solar production is plentiful, thus enabling reduced electricity consumption during evening peak hours. Similarly, these programs may incentivize households returning home to shift electric vehicle charging to nighttime hours when demand is lower.

However, real-world implementations of demand response programs have encountered
significant limitations in practice. In one of the most high profile applications of demand response, California's
state grid operator has deployed ``Flex Alert'' text message alerts to urge residents to reduce electricity use during peak demand periods,
especially during extreme heat
waves.\footnote{\url{https://www.latimes.com/california/story/2022-09-07/a-text-asked-millions-of-californians-to-save-energy-they-listened-averting-blackouts}}
While this method has at times effectively prompted immediate public action, it relies heavily on voluntary participation and may not
consistently achieve predictable load reductions. In particular, consumers showed evidence of messaging fatigue following repeated use of
Flex Alerts in 2022; and not all consumers were able to correctly implement some grid-supporting measures (such as pre-cooling) even
when they attempted to comply with a Flex Alert \citep{opiniondynamics2024}. As such, it may be implausible to hope to achieve large-scale,
systematic changes to consumer behavior via alert-based methods that rely on consumer goodwill alone.

At the other end of the spectrum, some jurisdictions have explored eliminating fixed-rate energy contracts and instead directly exposing
consumers to wholesale energy pricing.
For example, in 2017, a company called Griddy started offering consumers in Texas the option to
directly purchase energy at wholesale prices.\footnote{\url{https://www.caller.com/story/news/local/texas/state-bureau/2021/03/01/griddy-electric-bills-spiked-during-freeze-explanation/4549742001/}}
However, this system broke down when Texas faced electricity shortages during severe winter storms in
2021:\footnote{\url{https://www.nytimes.com/2021/02/20/us/texas-storm-electric-bills.html}}
Wholesale energy prices spiked $\sim$300 times higher than usual and Griddy was forced into bankruptcy
as many consumers were unable to pay their bills \citep{busby2021cascading}.
In the aftermath of Griddy's bankruptcy, consensus formed that exposing consumers
to unpredictable price fluctuations of this magnitude was not acceptable,
and Texas passed a law that prohibits directly exposing residential customers to wholesale energy prices \citep{TX_HB16_2021}.

The Electric Reliability Council of Texas (ERCOT), which oversees the Texan power grid,
has also designed a more targeted demand-response program, in which some loads are pre-registered as ``Emergency Response Service'' (ERS)
loads that can be curtailed---in exchange for payment---during energy shortages \citep{baldick2021variability}. Conceptually,
these ERS loads can then be understood as ``virtual power plants'' in the sense that, when the grid to faces an energy shortage, they can choose to either
increase generation capacity or pay the ERS to curtail \citep{mnatsakanyan2014novel}. Unfortunately, however, such
pay-to-curtail schemes appear in practice to be gameable, especially by market participants that are able to flexibly schedule (and
cancel) loads at peak hours.
For example, {\it The Economist} recently reported that some cryptocurrency companies were able to collect demand-response payments of the
same order of magnitude as their overall revenue:\footnote{\url{https://www.economist.com/united-states/2024/08/27/why-texas-republicans-are-souring-on-crypto}}
\begin{quote}
{\it ``On the hottest and coldest days, when demand for electricity peaks and the price rockets, the bitcoin miners sell power back to providers at a profit or stop mining for a fee, paid by
ERCOT. Doing so has become more lucrative than mining itself. In August of 2023 Riot\footnote{Riot Platforms is a publicly traded cryptocurrency company.}
collected \$32m from curtailing mining and just \$8.6m from selling bitcoin.''}
\end{quote}
Incentive-compatibility issues can also arise in more mundane settings. For example,
Pacific Gas \& Energy in California has experimented with a program called SmartAC
which compensates consumers for not using their air conditioners during system-wide
or local system strain \citep{pge2023smartac}; however, it is hard to stop such programs
from compensating consumers at times when they never intend to use their air conditioners
anyways (e.g., while they are traveling).



These challenges highlight the need for practical mechanisms for demand response that make it
financially rational for consumers to respond to demand-response signals, that do not subject
consumers to excessive risks or attention burdens, and that are simple and transparently non-gameable.
The goal of this paper is to outline such a system, and to establish its basic operating characteristics and
optimality properties in settings with many customers.
The core of our approach involves modeling consumer flexibility using an ``indifference set'' approach, whereby
each consumer has a pre-determined set of goals\footnote{The fact that we take the consumer indifference
sets as given reflects the fact that we are not trying to change consumer preferences (e.g., to accept
less comfortable indoor temperatures in order to save energy), but rather to leverage flexibility
inherent to existing preferences to achieve meaningful demand response.}
they want to have met (e.g., keep their indoor temperature
comfortable, have their electric vehicle charged by morning, etc.), but is indifferent to how these goals
are accomplished. We further assume that each consumer has access to a device, called a Home Energy Management
System (HEMS),\footnote{Devices of this type have been widely discussed in the literature
\citep{beaudin2015home, zhou2016smart}; and in fact many consumers already own smart thermostats
whose complexity is comparable to what we require of a HEMS \citep{stopps2021residential}, like Google Nest.}
that can receive price signals from the grid and rationally control devices such as to minimize
energy costs subject to meeting the consumer's goals.\footnote{Essentially, the HEMS needs to be
able to form and solve a linear program to plan optimal consumption.}
The key result of our paper is that, under this model and assuming that consumers are using HEMS to control their devices,
the grid can deploy pre-announced, time-varying energy prices to accomplish near optimal demand response under a mean-field limit.

The use of time-varying energy prices, also known as time-of-use (TOU) pricing, is one of the oldest ideas for
demand response, and discussions of TOU pricing go back to the 19th century \citep{hausman1984time}.
Several recent studies of TOU pricing are skeptical about their potential for meaningful demand
response \citep{schittekatte2024electricity}. Many of these limitations, however, are tied to specifics of existing TOU designs.
Most TOU programs that have been deployed to date have involved simple heuristics
(e.g., dividing each day into ``daytime'', ``evening'' and ``night'' and using constant electricity prices within
each of these 3 periods), and then relied on consumers to figure out themselves how to respond to these price signals
\citep{cosmo2014estimating}. There are limits to how much such simple TOU programs can shift demand. First of all, they
place considerable attention burden on consumers, and even when consumers do attempt to respond to these TOU prices
they will generally not have the sophistication to do so optimally \citep{allcott2011rethinking}. Second, if consumers were to procure technological
devices that enabled them to actually best-respond to such simple TOU prices, new unstable and/or undesirable demand patterns may emerge \citep{sunar2024electricity}: For
example, if everyone times their dishwasher to go on right after high ``evening'' prices turn to lower ``night'' prices, then
this will create a new demand spike the grid will need to address.

Meanwhile, in the academic literature, there has been considerable interest in deriving optimal
time-varying prices under generic choice models where each consumer has a (smooth, concave) welfare function
capturing their benefits from energy use \citep{datchanamoorthy2011optimal,yang2012game,peura2015dynamic,zhou2017incentive,adelman2019dynamic}.
Results in this space are of considerable conceptual interest (see the related work section for further discussion);
however, there are obstacles to putting the proposed algorithms directly to use in practice.
Asking consumers to explicitly reveal their welfare functions would require them to solve a complicated
reasoning task, and is thus likely not realistic in practice \citep{friedman1995complexity,hobman2016uptake}.
Additionally, reduced-form calibration of optimal pricing models requires measuring consumer demand elasticities
that may be hard to measure \citep{mnatsakanyan2014novel}.

The modeling assumptions we take in this paper, i.e., with indifference sets to capture flexibility and
HEMS to automatically control devices, considerably help alleviate issues with existing time-of-use pricing approaches.
We believe it to be far more practical and realistic to ask a consumer to set hard constraints on desired outcomes
than to elicit a complete welfare function; for example, it's easier for a consumer to say they'll accept
indoor temperature between 21--24\textdegree C during the day and 20--22\textdegree C at night than to specify
exactly how their experienced welfare depends on temperature over time. With HEMS, consumers can program their preferences once and avoid actively responding to price signals, all while preserving privacy by not sharing sensitive information with third parties.\footnote{As discussed further below, our learning results allow HEMS to
operate under a one-way communication model where the HEMS receives prices from the grid and reacts to them, but does not
send any information in response (other than the actual energy consumption).} Finally, automation allows the grid
to use prices that change frequently / gradually enough to avoid derived instability around price changes.

These assumptions, however, introduce novel analytical challenges. Under our indifference set model, consumer demand typically varies discontinuously with prices, rendering many technical tools from existing literature inapplicable \citep[e.g.,][]{adelman2019dynamic,yang2012game,zhou2017incentive}. The bulk of this paper demonstrates how classical results on convex analysis over random sets---building on the foundational work of \citet{aumann1965integrals} and \citet{artstein1975strong}---can be leveraged to address these challenges and establish the validity of our approach in large markets. Specifically, we show the following:
\begin{enumerate}
    \item We start by comparing our proposed pricing mechanism to a ``direct control mechanism''
    where the HEMS delegates full control of all user devices to the grid, thus enabling solution
    of a single global optimization problem that guarantees first-best demand response w.r.t. indifference-sets.
    We find that, under our model for demand response, the pricing mechanism described above is
    generally not optimal relative to the direct control mechanism, i.e., that a grid operator with direct control
    of all consumer devices can achieve better demand response (while respecting user indifference sets)
    that an operator who can only control demand via dynamic prices.\footnote{This
    finding may seem surprising at first glance, as there are
numerous welfare theorems in economics demonstrating that price-based mechanisms can achieve optimal welfare
in various systems \citep{mascolell1995microeconomic}. Our indifference-set model, however, is inconsistent
with a key assumption of these welfare theorems, often referred to as local non-satiation.
Local non-satiation implies that consumer demand moves smoothly with prices---but this is not
the case here, as in our setting we only consider hard constraints on the user side.}
    \item As the market size increases, the optimality gap of pricing mechanisms shrinks to zero.
    In the mean-field limit with infinitely many customers, the grid cost for both pricing and direct
    control mechanisms converge, and the price-based mechanisms becomes optimal.
    \item Mean-field optimal prices can be computed via a first-order algorithm that only requires querying HEMS about planned consumption profiles at different price levels, as these consumption profiles give information about a dual cost function. This allows for the straightforward application of first-order optimization methods to find the optimal price using only market-level data, without the grid needing to directly query for individual user preferences.
\end{enumerate}
Finally, we demonstrate practical promise of our pricing approach on a grid-simulation motivated
by pre-cooling with data from Phoenix, Arizona; we use OpenDSS to solve the relevant power-flow equations.
As seen in Figure \ref{fig: duck}, renewable production is heavily concentrated during daytime hours.
With flat rates, this results in the notorious ``duck curve'' with low (even negative) net demand during
the day that then quickly ramps up in the evening.\footnote{The ramp-up period
forms the neck of the ``duck'' and the evening peak demand its head \citep{denholm2015overgeneration}.}
In contrast, our proposed algorithm shifts demand earlier into the day when solar resources are plentiful,
thus enabling it to reduce load in the evening hours. We give further details on this experiment in
Section \ref{sec: exp}.

\begin{figure}[t]
   \begin{center}
        \includegraphics[width = 0.8\textwidth]{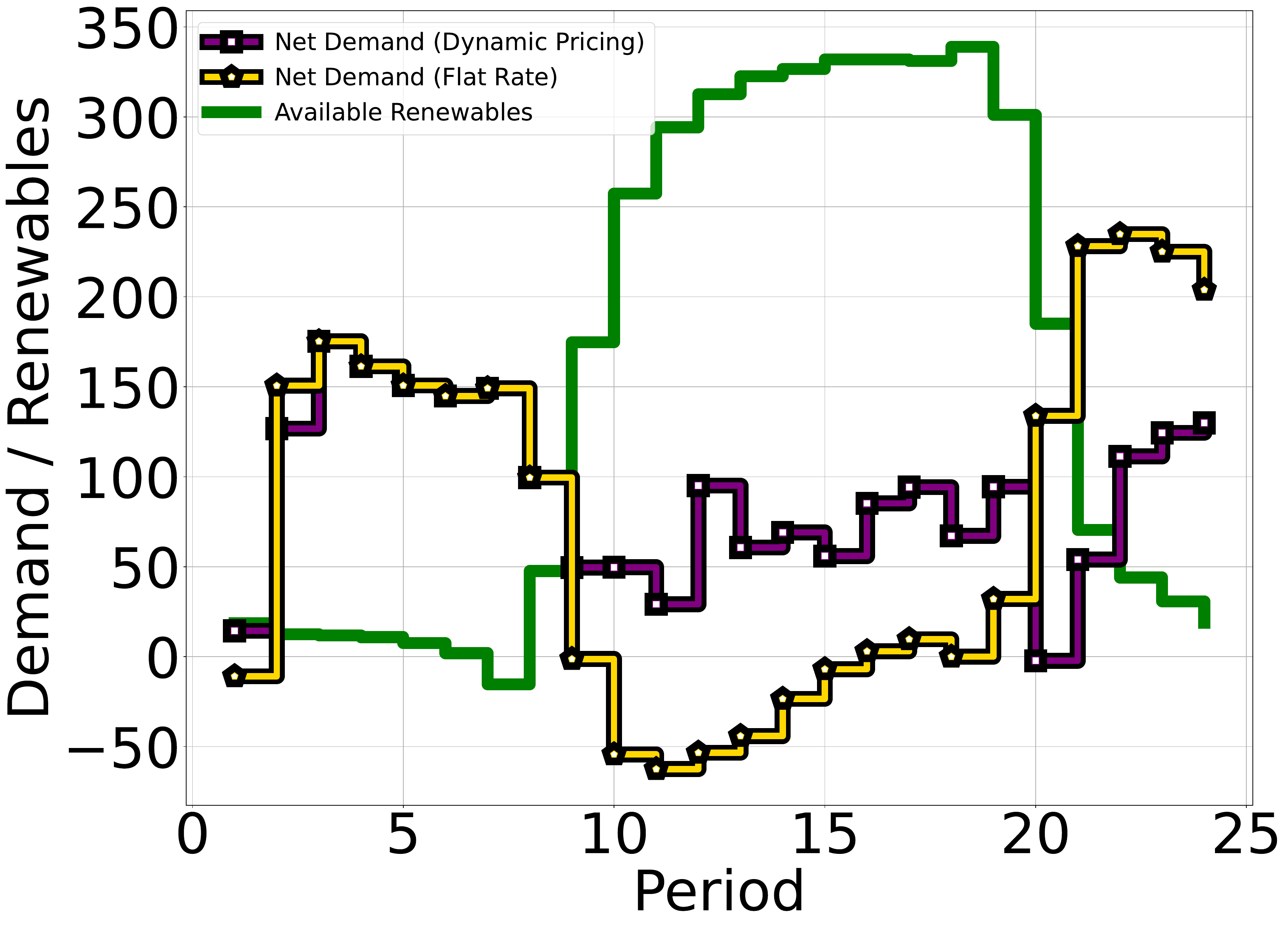}
        \end{center}
    \caption{Impact of optimized dynamic pricing on the net demand curve---demand minus available renewables---in the household cooling case study in Phoenix, Arizona (see Section \ref{sec: exp}). By transitioning from flat-rate pricing to the selected dynamic pricing, the classic ``duck-shaped'' net demand curve has been effectively flattened. The traditional duck curve, which represents net demand under flat-rate pricing, is problematic due to significant fluctuations from early morning (high demand/low renewables) to midday (low demand/high renewables) and again from midday to late evening(high demand/low renewables). The dynamic pricing applied in this study was derived solely from querying HEMS regarding their consumption profiles in response to various pricing signals.}\label{fig: duck}
\end{figure}

\subsection{Related work}

There is a rich literature on using time-varying prices for demand response
\citep{adelman2019dynamic,datchanamoorthy2011optimal,peura2015dynamic,schittekatte2024electricity,uckun2012dynamic,yang2012game,zhou2017incentive,zhou2019online}.
Formal studies of such pricing-based mechanisms are typically framed as Stackelberg games with a two-level leader-follower
structure \citep{adelman2019dynamic,yang2012game,zhou2017incentive,zhou2019online}: First the grid sets prices, and then consumers
respond to these prices. Our indifference-set approach is a special case of such a Stackelberg games, where lower-level
optimization involves consumers choosing their energy consumption to minimize costs while remaining within their indifference
sets (i.e., the consumer preferences are instantiated as hard constraints on allowed consumption patterns). We note that
Stackleberg games are also often used for the more abstract task of policy learning with strategic agents
\citep{chen2020learning,munro2024treatment,perdomo2020performative,sahoo2022policy};
and our paper can be seen as a part of this broader literature that exploits specific structure of
consumer energy price response.

The major difference between our paper and existing work is that we model consumer preferences
using constraints (as captured by indifference sets) rather than using smooth welfare
functions---and as a result of that this requires us to develop entirely different technical notions based on convex analysis on random sets
to study relevant system properties (as existing analytic strategies fundamentally rely on the smoothness
of welfare functions). As discussed above, we believe our indifference-set approach
to be substantially more tractable from a practical implementation point of view.

One recent and representative paper from the existing literature is
\citet{adelman2019dynamic}, who study dynamic pricing in a setting that includes current energy
demand, appliance status, operation times, global variables such as weather conditions, among
many others. They assume that each consumer's welfare depends on their energy usage in a differentiable
way, and that consumers choose their energy consumption to optimize welfare in response to set prices (e.g.,
using a HEMS). Given this setup, the authors define the social planner's welfare function as the sum
of all consumers' welfare minus the grid's generation cost, and seek prices that optimize the social
planner's welfare under the posited consumer behavior model (the consumers choose how much to consume
in response to prices, and the grid must then meet this demand). Their main result is that, as the
system size grows large, the optimal price mechanism in their model is one that perfectly externalizes
the grid's incremental production costs to the consumers, i.e., in equilibrium, energy is priced to
match the marginal cost of production. The authors propose finding prices that maximize the social
planner's welfare via a tempered fixed-point iteration algorithm.

Like \citet{adelman2019dynamic}, we consider a model where---at least in a qualitative sense---prices are
set in the view of mutual benefit of consumers and producers. However, unlike \citet{adelman2019dynamic}, we
do not rely on properties of the consumer welfare function, and do not posit a social planner who is able
to trade-off consumer and producer welfare on an additive scale.
Then, several mathematical tools used by \citet{adelman2019dynamic}---such as the implicit function theorem
to capture consumer response to prices---are no longer applicable in our model where consumer preferences
are modeled using indifference set, and so we cannot use analytic strategies from their paper to study our setting.

In another recent line of work, \citet{zhou2017incentive} consider demand-response programs that employ price incentives, where each consumer aims to minimize their own cost function. This cost function is composed of two parts: One that reflects power consumption and remains independent of price incentives (e.g., minimizing deviation from a nominal consumption level), and another that accounts for the expenses incurred due to price incentives. Building on this framework, \citet{zhou2019online} extends the analysis to incorporate discrete decision variables and dynamically coupled devices over time, such as HVAC systems. Their approach shares similarities with ours, particularly in modeling consumer objectives as minimizing cost functions and solving a social-welfare optimization problem (the direct method in our setup) to determine optimal pricing strategies. 
However, the authors fundamentally rely on an assumption that consumers set their power consumption by minimizing strongly convex cost functions whereas the cost functions implied by
our indifference-set approach are not strongly convex---thus again making their formal results non-applicable to our setting. We also note that the lack of strong convexity in our
model results in different qualitative system behavior: In the setting of \citet{zhou2017incentive} the authors find pricing to always achieve optimal demand response, whereas
in our model pricing is not optimal in finite systems (and the optimality gap only closes in a large-sample limit).



Finally, there have also been a number of recent papers studying incentive- rather than price-based mechanisms for demand response.
\citet{agrawal2022design}, \citet{daniels2014demand}, and \citet{webb2019coordinating} explore direct load control contracts, where utility companies manage customer loads, typically to reduce consumption as a form of demand response (in contrast, in the pricing algorithm we evaluate in our experiments, the grid has no explicit control over users' consumption; instead, HEMS selects optimal consumption profiles that meet users' needs while minimizing costs).
\citet{fattahi2023peak} formulates a stochastic dynamic program to design optimal contracts and develop approximation methods to achieve near-optimal solutions for energy consumption (they avoid explicitly referring to these contracts as demand response due to their complexity and the extended nature of the control periods). 
\citet{allcott2011social} studies the effect of ``home energy reports'' that compare a consumer's energy use to that of their neighbors, and finds that such programs can meaningfully shift behavior.

Our analytic results rely on random set theory to characterize consumer flexibility in demand response programs. Random sets have a rich history in applied probability, going back to \citet{aumann1965integrals}. In particular, the strong laws of large numbers for random of \citet{artstein1975strong} plays a key role in our analysis.
One key early application of these methods was to show existence of competitive equilibria in markets
with a continuum of agents who may have non-concave utility functions \citep{aumann1966existence,hildenbrand1974core}.
More recently, random-set methods have been used in econometrics to study inference in partially
identified models \citep{beresteanu2008asymptotic,beresteanu2012partial,kaido2014asymptotically}.
Textbook treatments of random set theory are given in \citet{molchanov2005theory} and \citet{artstein2015asymptotic}.

Our results on learning-to-price only assume that the grid can observe energy consumption levels under set prices;
and in particular do not require HEMS to share any (potentially sensitive) private information (e.g., information on
consumer preferences that underlie its consumption choices). As such, our learning problem can be seen as a bandit
problem with a continuous, high-dimensional action space \citep{ferreira2018online}.
In general, bandit problems with continuous action spaces do not allow for last learning \citep{shamir2013complexity}.
However, in our setting, we show that the grid can use mean-field properties of the system to effectively
infer relevant gradients from raw consumption data, thus enabling faster convergence properties \citep{wager2021experimenting}.

The remainder of the paper is structured as follows. Section 2 introduces the optimal demand response problem and formalizes the indifference set model. Section 3 establishes the mean-field optimality of pricing mechanisms and characterizes their asymptotic performance. Section 4 presents a first-order algorithm to compute optimal prices using only market-level data. Section 5 demonstrates the effectiveness of our approach through a case study on pre-cooling in Phoenix, Arizona. Finally, Section 6 discusses broader implications and concludes.

\section{Optimal Demand Response}

To design effective demand response mechanisms, we first establish a framework that captures consumer flexibility. We introduce the concept of indifference sets, which define the range of consumption profiles consumers find acceptable, and formulate the optimization problem faced by both the grid and consumers.

We consider a model with $j = 1, \, \ldots, \, d$ time periods, e.g., $d=24$ for hourly
planning over the course of a day, and with $i = 1, \, \ldots, \, n$ consumers. Each
consumer has certain needs that must be met---for example, maintaining indoor temperatures within acceptable bounds or ensuring that an electric vehicle (EV) is charged overnight---and they will consume electricity $q_i \in \RR^d$ to meet these needs. The grid operator
has access to a variety of energy resources, such as gas, wind and/or solar plants, that can be
used to produce this energy. We assume that the renewable energy resources, such as wind
and solar, exogenously produce energy $Q_0 \in \RR^d$ at zero marginal cost; the grid operator then
needs to deploy non-renewable resources to meet the remaining demand $\sum_{i = 1}^n q_i - Q_0$.

The goal of a demand-response program is to shift consumer demand so that the grid can
produce the needed energy more efficiently, and in particular so that it can avoid spikes in
net demand for non-renewable energy. We formalize the underlying problem as follows:

\begin{defi}
The {\bf indifference set} $\bR_i \subset \RR^d$ of each consumer is the set of
all energy consumption profiles that allow the consumer to meet their pre-specified needs.
\end{defi}

\begin{defi}
The {\bf optimal demand response} problem involves implementing acceptable energy
consumption profiles for each consumer such as to optimize a grid objective. If
the grid has a cost function $\sigma(Q, \, Q_0)$ for producing energy $Q \in \RR^d$
given renewable production $Q_0$, then the optimal demand response problem requires
finding a solution { to the Direct Method (DM) risk value}
\begin{equation}
\label{eq:opt_dr}
\hopt_n^{\myom} = \min\left\{\sigma\left(\sum_{i = 1}^n q_i, \, Q_0\right) : q_i \in R_i \text{ for all } i = 1, \, \ldots, \, n \right\}.
\end{equation}
\end{defi}

We give some idealized examples of possible indifference sets below. Throughout, we assume
that we are planning energy consumption over a 24 hour period starting at midnight so,
e.g., $q_{i1}$ is the consumption in kWh of the $i$-th user between midnight and 1am, etc.
The examples below are intentionally simple; and we emphasize that our formal results
allow full flexibility in specifying these indifference sets. We will not require any specific
parametrizations or functional forms for these sets; rather, we will simply assume that
each consumer has a HEMS that understands their own acceptable energy consumption profiles.

\begin{example}
Consider a consumer who plugs in an electric vehicle at midnight and only needs to add 50 kWh of charge by 7 am. Their indifference set is then $\bR_1 = \{q \geq 0 : \sum_{j = 1}^7 q_j = 50, \, q_j = 0 \text{ for all } j > 7\}$.
\end{example}

\begin{example}
Consider a consumer who only needs to run their dryer between 8 and 11 AM, which requires 3 kWh
over a 1 hour period. Their indifference set is then
$\bR_2 = \{q \geq 0: \sum_{j = 8}^{11} q_j = 3, \, \text{where } q_j = 0 \text{ or } 3 \text{ for all } 8 \leq j \leq 11, \, \text{and } q_j = 0 \text{ else}\}$.
\end{example}

\begin{example}\label{ex: heat-dynamics}
Consider a consumer who only wants to operate an air conditioner to maintain the indoor temperature of
their house within an acceptable range. Following \citet{li2011optimal}, we model indoor temperature
dynamics using the following linear formulation motivated by the heat equation,
\begin{align}\label{eq: linear-heat}
T^{\mathsf{in}}_t = T^{\mathsf{in}}_{t-1} + \alpha (T^{\mathsf{out}}_t - T^{\mathsf{in}}_{t-1}) + \beta q_t,
\end{align}
where $ \alpha\in \RR$ captures the insulation properties of the house, and $\beta < 0$
represents the cooling efficiency power of the system. Writing $[T^{\mathsf{min}}, \, T^{\mathsf{max}}]$
for the acceptable temperature range, and $T^{\mathsf{in}}_0$ for the initial indoor temperature,
the consumer's indifference set then becomes:
\begin{equation*}
\bR_3 = \cb{q \geq 0 :  T^{\mathsf{min}} \leq \sum_{i = 1}^j \p{1 - \alpha}^{j - i} \p{\beta q_i + \alpha T_i^{\mathsf{out}}} + (1 - \alpha)^j \, T^{\mathsf{in}}_0 \leq T^{\mathsf{max}} \text{ for all } 1 \leq j \leq 24}.
\end{equation*}
\end{example}

\begin{example}
Consider a consumer who wants to operate all 3 devices described above. Then
their indifference set is  $\bR_4 = \bR_1 \oplus \bR_2 \oplus \bR_3$.\footnote{The Minkowski sum (vector sum) of two indifference sets defined as $\bR_1\oplus \bR_2=\{a+b:a\in \bR_1,b\in \bR_2\}$.}
\end{example}

Our main goal of this paper is to characterize solutions to the optimal demand
response problem, and to design algorithms for achieving optimal consumption.
For flexibility, we will not seek to model individual consumer devices directly;
rather, we will assume throughout that we interact with consumers via a HEMS
device. This assumption is in line with current practice \citep{beaudin2015home,zhou2016smart}.

\begin{defi}
\label{defi:HEMS}
A {\bf home energy-management system (HEMS)} serving consumer $i$
is a device that can
\begin{enumerate}
\item Compute the indifference set $\bR_i$ in terms of the consumer's
high-level preferences and device properties;
\item Communicate with the grid operator; and
\item Choose and implement a consumption profile $q_i \in \bR_i$.
\end{enumerate}
\end{defi}

At this stage, our framework remains abstract: For example, our definitions
so far, it would, in principle, allow for the HEMS to give the grid operator direct
control of all consumer devices, and let the grid solve the optimization problem
\eqref{eq:opt_dr}. This approach would, of course, achieve optimal demand response,
but would likely not be implementable in real-world systems.
A more practical alternative is to decouple decision making by the grid and by the
HEMS via a price mechanism: The HEMS optimizes consumption to minimize energy costs
under (potentially time-varying) prices, and the grid chooses prices such as to induce
desirable consumption profiles.

\begin{defi}
\label{defi:HEMS_price}
A {\bf price-responsive HEMS} is a HEMS as in Definition \ref{defi:HEMS} that receives
a price-vector $p \in \RR^d$ from the grid, and then chooses
\begin{equation}\label{eq:  smart-device}
q_i := q_i(p) \in \argmin_q\cb{p \cdot q : q \in \bR_i}.
\end{equation}
\end{defi}

In a system where consumers operate price-responsive HEMS, the grid's optimal pricing
problem becomes a bi-level optimization problem \citep{zhou2017incentive}:
\begin{equation}
\label{eq:opt_price}
p^* \in \argmin_p\cb{\sigma\left(\sum_{i = 1}^n q_i(p), \, Q_0\right)},
\end{equation}
resulting in an objective value $\hopt_n^{\mydp}$  as the optimal Price-based Method (PM) risk value. 
We are now ready to state our main research questions:
First, can a pricing mechanism as in \eqref{eq:opt_price} solve the optimal demand
response problem \eqref{eq:opt_dr}, or will systems with pricing-based HEMS fall
short of what could be achieved via direct device control by the grid?
Second, if an optimal pricing mechanism exists, are there computationally efficient and practically
applicable algorithms to learn the optimal prices?
Overall, we will find that while price-based mechanisms may be suboptimal in finite systems,
optimality gap of pricing as in \eqref{eq:opt_price} vanishes as the system size grows.
Furthermore, optimal prices can be recovered via a simple first-order algorithm.

\begin{rmk}
In our model, consumer indifference sets remain exogenous to prices, i.e., consumers
respond to prices by optimizing their consumption with $\bR_i$ as in \eqref{eq:  smart-device},
but the $\bR_i$ themselves remain fixed. This modeling choice reflects our underlying research
question: How can we design mechanisms that leverage {\it existing consumer flexibility} for demand
response (and how does a price-based mechanism compare with direct control)? Once the $\bR_i$
are fixed and coded into HEMS, optimal demand response essentially becomes an optimization problem,
and our challenge is to find (or approximate) the optimum in a feasible and pragmatic way (i.e.,
in a way that respects natural communication and privacy desiderata). Reasoning about
how the $\bR_i$ themselves may change in the long run in response to new pricing strategies
is an interesting question (see Section \ref{sec:discussion} for further discussion); however,
doing so would likely not be automatable using HEMS (and would instead require a deeper understanding
of how consumers update preferences over time), and so is outside the scope of the present analysis.
\end{rmk}

\begin{rmk}
As a corollary to the fact that the indifference sets $\bR_i$ are exogenous with respect to prices
we observe that, in our model, price-responsive HEMS are only responsive to variation in the
shape of the price curve and not in the overall level of the price, i.e., $q_i(p) = q_i(\lambda p)$
for any $\lambda > 0$. Similarly, optimal prices in \eqref{eq:opt_price} are scale invariant:
If $p^*$ is optimal, then $\lambda \, p^*$ is also optimal for any $\lambda > 0$.
Thus, in our model, demand response is purely about cross-time substitution of energy
use, and successful price-based demand response involves understanding how the
the relative magnitudes of the price signal over time will shape demand.
\end{rmk}

\begin{rmk}
Because behavior is scale-invariant to prices, the overall price level remains a free
parameter in our model---and can be set by the regulator and/or the utilities. In other
words, once our approach outputs a price curve $p^*$, the regulator can rescale it to
$\lambda \, p^*$ for any $\lambda > 0$ in order to control some aggregate metrics of energy
costs or utility revenue; see the discussion following Theorem \ref{thm:stable}
for more details.
\end{rmk}

\subsection{A Warm-up Illustration}
Before attempting a general answer to these questions, we begin by examining
both direct and price-based approaches to demand response in a simplified scenario
where the problems \eqref{eq:opt_dr} and \eqref{eq:opt_price} can be solved
explicitly. We consider a two-period setting with no renewable production
(i.e., with $Q_0 = 0$), and where the grid wants to minimize peak demand (and thus
to use demand-response programs to shift consumer demand away from peak periods);
the grid cost functions is $\sigma(Q) = \sup\cb{Q_j : j = 1, \, \ldots, \, d}$.
Finally, we assume that each consumer has a linear indifference set
\begin{equation}
\label{eq:example_lin}
\bR_i = \cb{q \in \RR_+^2 : q_1 / a_i + q_2 / b_i = 1},
\end{equation}
where $a_i, \ b_i > 0$ are consumer-specific parameters.

\begin{propo}
\label{propo:warmup}
Consider the above setting with $i = 1, \, \ldots, \, n$ users with indifference
sets as in \eqref{eq:example_lin}, and where the grid seeks to minimize peak-load.
Then, optimal demand response as in \eqref{eq:opt_dr} achieves an objective value
\begin{equation}
\hopt_n^{\myom} = \max_{0 \leq z \leq 1}\cb{L_n(z)}, \ \
L_n(z) =  \sum\limits_{i=1}^{n}\p{ (1-z) b_i \, \ind\p{(1-z) b_i< z a_i }+ z a_i \, \ind\p{(1-z) b_i \ge z a_i }}.
\end{equation}
Meanwhile, optimal pricing as in \eqref{eq:opt_price} achieves
\begin{equation}
\hopt_n^{\mydp} = \min_{0 \leq z \leq 1}\cb{U_n(z)}, 
\ \ U_n(z)=\min\cb{U^+_n(z),U^-(z)}\,,
\end{equation}
where $U^+_n(z),U^-(z)$ are given by
\begin{align*}
    U^+_n(z) &=  \max\cb{\sum\limits_{i=1}^{n} a_i \, \ind\p{a_i z \leq b_i (1 - z) }, \, \sum\limits_{i=1}^{n} b_i \, \ind\p{a_i z > b_i (1 - z) }}\,,\\
    U^-_n(z) &=  \max\cb{\sum\limits_{i=1}^{n} a_i \, \ind\p{-a_i z \leq b_i (-1 + z) }, \, \sum\limits_{i=1}^{n} b_i \, \ind\p{-a_i z > b_i (-1 + z) }}\,.
\end{align*}
\end{propo}

Because direct optimization is guaranteed to achieve optimal demand response, we know that
we must have
$$ \max_{0 \leq z \leq 1}\cb{L_n(z)} \leq \min_{0 \leq z \leq 1}\cb{U_n(z)}, $$
and furthermore, the size of the vertical gap between these functions $L_n(z)$
and $U_n(z)$ can be used to measure the sub-optimality of the pricing mechanism.
To gain further insight into this gap, we simulate consumers and draw the indifference
set parameters in \eqref{eq:example_lin} as
\begin{equation}
a_i, \, b_i \overset{\text{iid}}{\sim}\unif[0, \, 2].
\end{equation}
In Figure \ref{fig:ell-infty-gap}, we plot the resulting $n^{-1}L_n(z)$ and $n^{-1}U_n(z)$
curves for both a realization of a small system with $n = 10$ consumers and a
realization of a larger system with $n = 100$ consumers.

\begin{figure}[t]
   \begin{center}
        \includegraphics[width = 0.5\textwidth]{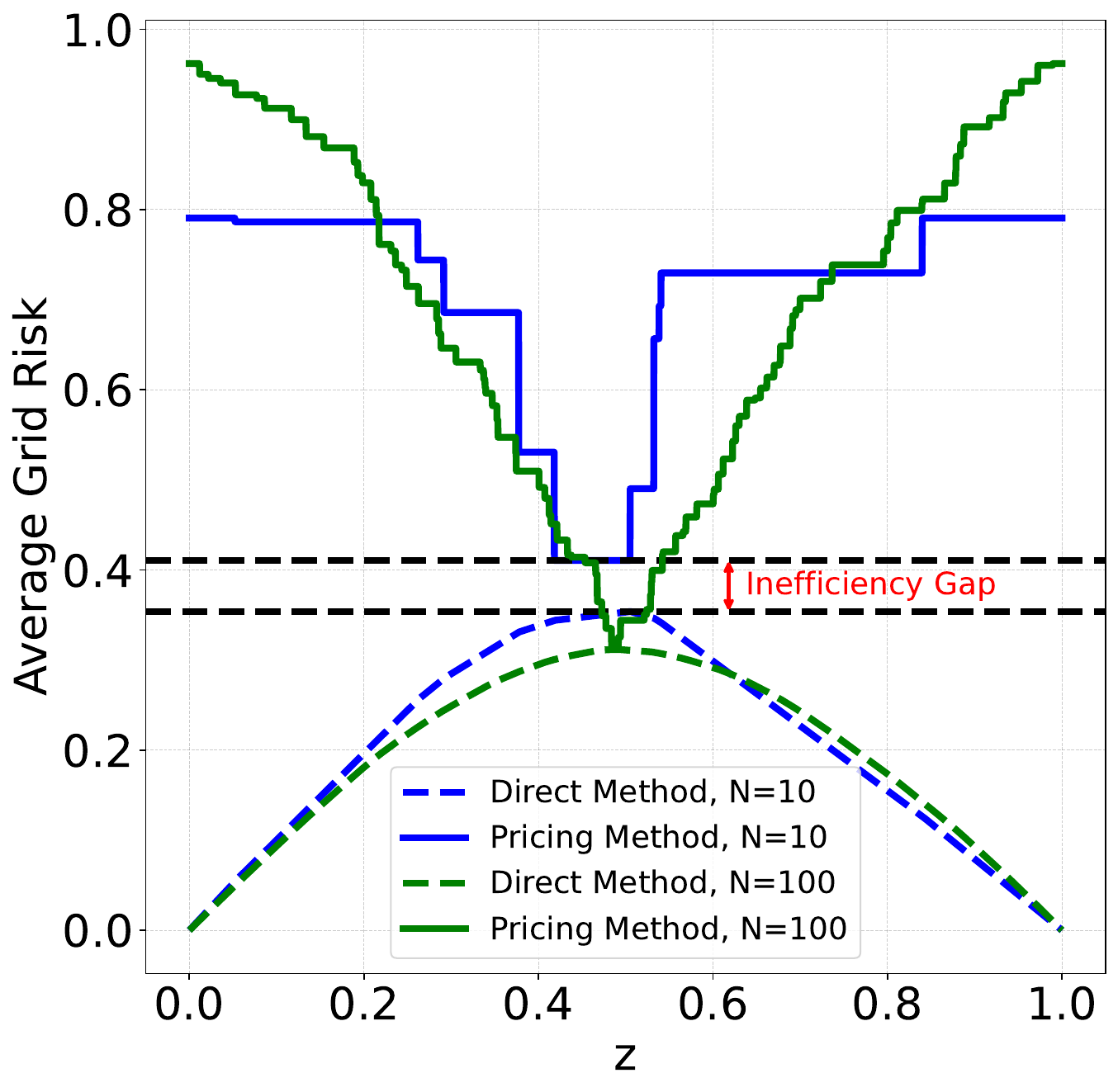}
        \end{center}
    \caption{Peak demand achieved via the direct control and pricing mechanisms
    in the setting of Proposition \ref{propo:warmup}, in both a small system with
    $n = 10$ consumers and a larger system with $n =100$ consumers. For direct control,
    we plot $n^{-1} L_n(z)$ (and the scaled optimal objective value is the maximum of this function),
    while for pricing we plot $n^{-1} U_n(z)$ (and the scaled optimal objective value is the minimum
    of this function).}
    \label{fig:ell-infty-gap}
\end{figure}

We immediately see that, in the small system, pricing is suboptimal:
Direct control achieves scaled peak demand of 0.35, while optimal
pricing can only get to a scaled peak demand of 0.41; in other
words, the grid cost with optimal pricing is 17\% higher with pricing
than with direct control. However, this the optimality gap of pricing
essentially vanishes as we scale the system up to $n = 100$ customers;
visually we already see no gap in Figure \ref{fig:ell-infty-gap} at
this scale. Below, we will prove that this finding was no accident---and
that, while potentially suboptimal in small systems, pricing mechanisms
are quasi-optimal in large systems under considerable generality.

\section{Mean-field Optimality of Dynamic Pricing}
\label{sec:theory}

With the indifference set framework in place, we examine whether price-based mechanisms can achieve optimal demand response. We show that while pricing is generally suboptimal in finite markets, it converges to optimality in a mean-field setting as the number of consumers grows.

To present our main result on large-sample optimality of pricing mechanisms for
demand response, we first need to introduce some notation.
Let $\mathfrak{C}_d$ denote the space of all non-empty compact subsets of $\RR^d$.
For two sets $C_1, \, C_2 \in \mathfrak{C}_d$, we write their Minkowski sum as
\[
C_1\oplus C_2=\{c_1+c_2: c_1\in C_1, c_2\in C_2\}. 
\]
For any non-empty subset $C$ of $\reals^d$, we define its support function
as
\begin{equation}
\delta(s; \, C) = \sup\cb{c \cdot s : c \in C}. 
\end{equation}
We measure set similarity using the Hausdorff distance
\[
\sH(C_1,C_2)=\inf\{\eps>0:  C_1\in C_2\oplus B_\eps(0)\,, C_2\in C_1\oplus B_\eps(0)\},
\]
where $B_r(x) \subset \RR^d$ denotes the $d$-dimensional Euclidean ball of radius
$r$ centered at $x$.
Recall we have assumed that each user has an indifference set $\bR \subset \RR^d$;
for our analysis, we will take these indifference sets to be random, and we construct
random sets $\bR \in \mathfrak{C}_d$ as follows. Given a probability space
$(\Omega, \, \mathcal{F}, \, P)$, we say that a multifunction $R : \Omega \To \mathfrak{C}_d$
is measurable if, for all open sets $O \subset \RR^d$ 
$$ R^{-1}(O) = \cb{\omega \in \Omega : R(\omega) \cap O \neq \emptyset} $$
is measurable \citep[Chapter 14]{rockafellar2009variational}.
We say that a random set is measurable if $\bR = R(\omega)$ for a measurable
multifunction.\footnote{Some papers cited below use different definitions of measurability
for random sets; see \citet{bellay1986measurability} for results on the equivalence of these definitions.}
We further define a selection of a random set $\bR$ as any random variable $X \in \RR^d$ with \citep{vitale1990brunn}
\begin{equation}
\label{eq:selection}
\prob(X \in \bR) = 1, \ \ \ \ \EE{\lVert X\rVert} < \infty,
\end{equation}
i.e., a random variable that always chooses an element from the random set;
and define the expected value of a random set as
\begin{equation}
\EE{\bR} = \cb{q = \EE{X} : X \text{ is a selection of } \bR},
\end{equation}
while noting that if $\E[\|\bR\|] < +\infty$, then $\E[\bR]$ is non-empty and 
$\E[\bR] \in \fC_d$ \citep{aumann1965integrals}. This notion of a random
set results in the following strong law of large numbers:
If $\E[\|\bR\|] < \infty$,
then \citep{artstein1975strong}
\begin{equation}\label{eq: as-hasudorf}
\lim\limits_{n\to \infty}^{}\sH \p{\frac{1}{n}\bigoplus\limits_{i=1}^{n} \bR_i , \ \EE{\conv(\bR)}} =_{a.s.} 0,
\end{equation}
where $\conv$ denote the convex hull operation.\footnote{Note that since $\conv(\fC_d)$ is
a closed subset of $\fC_d$ and $\conv$ is a continuous mapping with respect to the
Hausdorff distance in a finite-dimensional space, $\conv(\bR)$ is measurable, and
its expectation is well-defined.}
We are now ready to state our main assumptions. Our first assumption is a standard
technical assumption, and implies useful concentration properties for random sets.

\begin{assumption}\label{assum: R-basic}
The user indifference sets are independent and identically from a
distribution over $\mathfrak{C}_d$; in more detail, we assume that
there is a non-atomic probability space $(\Omega, \, \mathcal{F}, \, P)$ and a
measurable multifunction $R : \Omega \To \mathfrak{C}_d$ such that
$\bR = R(\omega)$. Furthermore, we assume that $\E[\|\bR\|]<+\infty$.
\end{assumption}


Next, we introduce an assumption the grid risk functions. In real-world applications, practitioners often aim to regulate peak demand, total consumption, or sharp transitions between periods. We find that a class of order-one positively homogeneous grid risk functions is sufficiently powerful to capture these behaviors and effectively shape the net-demand curve.

\begin{assumption}\label{assum:sigma}
The grid costs can be represented in terms of a convex function $\rho:\reals^d\to \reals$ of net demand, i.e.,  $\sigma(Q, \, Q_0)=\rho(Q-Q_0)$. Furthermore, $\rho(.)$ is order-one positively homogeneous, i.e., for every $\alpha \ge 0$ and $q\in \reals^d$ we have that $\rho(\alpha q)=\alpha \rho(q)$.
\end{assumption}


We now state some technical results that highlight key consequences of these
assumptions. First, under our setting, set-expectation commutes with the user
cost function in the following sense:

\begin{lemma}
\label{lemm:H}
Under Assumption \ref{assum: R-basic}, write the user-$i$ cost function from solving \eqref{eq:opt_price} as
\begin{equation}
\label{eq:hp}
h(p; \, \bR_i)=\min\cb{q \cdot p : q\in \bR_i};
\end{equation}
and let $H(p) = h(p; \, \EE{\bR_i})$.
Then, for any $p \in \RR^d$ we also have $H(p) = \EE{h(p; \, \bR_i)}$.
\proof
The user cost function can also be written as a support function, i.e.,
$h(p; \, \bR_i) = -\delta(-p; \, \bR_i)$; thus, proving our desired result
is equivalent to showing that
$$ \EE{\delta(p; \, \bR_i)} =\delta(p; \,  \EE{\bR_i}) $$
for all $p$ in $\RR^d$. We now note 3 well-known facts in convex analysis:
{\it (1)} For any random set $\bf{C}$ that is almost surely convex, we have
\citep{artstein1974calculus}
\begin{equation*}
\E[\delta(u; \, \bC)] = \delta(u; \, \E[\bC]) \text{ for all } u \in \RR^d;
\end{equation*}
{\it (2)} Whenever $\bR_i$ has a non-atomic distribution,
we have $\EE{\conv(\bR_i)} = \EE{\bR_i}$ \citep{aumann1965integrals};
and {\it (3)} the support function of a set and its convex hull are equal, i.e.,
$\delta(.; \, \bR) = \delta(.; \, \conv(\bR))$. It then follows that
$$ \EE{\delta(p; \, \bR_i)} 
\overset{\text{(3)}}{=} \EE{\delta(p; \, \conv(\bR_i))} 
\overset{\text{(1)}}{=} \delta(p; \,  \EE{\conv(\bR_i)})
\overset{\text{(2)}}{=} \delta(p; \,  \EE{\bR_i}), $$
as claimed.
\qed
\end{lemma}

Second, under Assumption \ref{assum:sigma}, there exists a compact convex set
$\mathcal{P} \subset \mathbb{R}^d $ such that $ \rho(.)$ is its support function
\citep[e.g.,][Claim A.2.24]{artstein2015asymptotic}:\footnote{See Lemma
\ref{lemma: rho-ball} in the Appendix for a self-contained proof of this claim.}
\begin{align}\label{eq: cP}
\rho(q) = \delta\p{q; \, \mathcal{P}}, \ \ \ \
\mathcal{P} = \left\{ z \in \mathbb{R}^d : z \cdot q \le \rho(q)\,,\quad \forall q \in \mathbb{R}^d \right\}.
\end{align}
Given these technical preliminaries, we are now ready to state the first characterization result.

\begin{thm}\label{thm: zero-gap-basic}
Under Assumptions \ref{assum: R-basic} and \ref{assum:sigma}, suppose
we have a sequence of systems with renewable production scaling with
system size, $\lim_{n \rightarrow \infty} Q_0 / n =_{a.s.} q_0$. Then,
\begin{enumerate}
\item\label{part-i-basic} The scaled cost of the direct control mechanism converges almost surely,
\begin{equation}\label{eq: mean-field-DM}
\lim\limits_{n\to \infty} \frac{1}{n} \hopt_n^{\myom} = \mathsf{R}^{\myom} := \max \cb{H(z)-z \cdot q_0 : z\in \cP},
\end{equation}
where $\cP$ is given in \eqref{eq: cP}, and $H(.)$ is as in Lemma \ref{lemm:H}.
\item \label{part: 2-basic} 
 For any price $p$, if the mean-field cost function $H(.)$ is differentiable at $p$, then 
 the scaled grid cost from using prices $p$ converges almost surely, 
\begin{equation}\label{eq: mean-field-PM}
\lim\limits_{n\to \infty} \frac{1}{n} \sigma\p{\sum_{i = 1}^n q_i(p), \, Q_0} = \rho\p{\nabla H(p) - q_0}.
\end{equation}
\end{enumerate}
\proof
We start with part 1. Let
\[
\widehat{\bR}_n=\frac{1}{n} \bigoplus\limits_{i=1}^{n}\bR_i
\]
denote the Minkowski average of the user indifference sets. By Assumption \ref{assum:sigma},
\begin{equation}
\label{eq: DM-alternative}
\begin{split}
\frac{1}{n} \hopt_n^{\myom}
&= \min \cb{\frac{1}{n} \sigma\left(\sum\limits_{i=1}^{n}q_i, \, Q_0\right) : q_i\in \bR_i \text{ for all } 1 \leq i \leq n} \\
&= \min\cb{ \rho\p{\frac{1}{n}\sum\limits_{i=1}^{n}q_i - Q_0/n} : q_i\in \bR_i \text{ for all } 1 \leq i \leq n}\\
&= \min\cb{ \rho\p{q - Q_0/n} : q \in \widehat{\bR}_n}.
\end{split}
\end{equation}
Thus, by the strong law of large numbers \eqref{eq: as-hasudorf}, the almost sure
convergence of $Q_0/n$ and Lipschitz-continuity of $\rho$, we have
\begin{equation}
\lim_{n \rightarrow \infty} \frac{1}{n} \hopt_n^{\myom} =_{a.s.} \min\cb{ \rho\p{q - q_0} : q \in \EE{\bR}}. 
\end{equation}
Now, by the characterization of $\rho(\cdot)$ given in \eqref{eq: cP}, we have
$$ \min_{ q \in \EE{\bR}}\cb{ \rho\p{q - q_0}} 
= \min_{q \in \EE{\bR}}\cb{ \max_{z \in \cP}\cb{ z \cdot \p{q - q_0}}}. $$
Furthermore, because we are working in a non-atomic probability space and $\EE{\lVert \bR\rVert} < \infty$,
the set $\EE{\bR}$ must be convex and compact \citep{aumann1965integrals},
and so by the minimax theorem
\begin{align*}
\min_{q \in \EE{\bR}}\cb{ \max_{z \in \cP}\cb{ z \cdot \p{q - q_0}}}
= \max_{z \in \cP}\cb{ \min_{q \in \EE{\bR}}\cb{ z \cdot \p{q - q_0}}} 
= \max_{z \in \cP}\cb{ H(z) - z \cdot q_0},
\end{align*}
which establishes \eqref{eq: mean-field-DM}.

Moving to part 2, we start by noting that---in our model---the user energy
consumption $q_i(p)$ is a subgradient of the user cost $h(p; \, \bR_i)$, i.e.,
$q_i(p) \in \partial h(p; \, \bR_i)$. To verify this claim, we
note that $h(p;\bR)$ is concave because it is a point minimization of concave functions $p \cdot q$,
and $h(p; \, \bR_i)=p \cdot q_i(p)$. Now, for any $p'$,
 \begin{align*}
q_i(p) \cdot p'&= q_i(p) \cdot p+(p'-p) \cdot q_i(p) =  h(p; \, \bR_i)+ (p'-p) \cdot q_i(p).
 \end{align*}
Writing $q_i(p')$ for optimal consumption under $p'$, we must have
$p' \cdot q_i(p')\le p' \cdot q_i(p)$, and so
\begin{equation}\label{eq: q_p_subgradient}
h(p'; \, \bR_i)=p' \cdot q_i(p') \le p' \cdot q_i(p) = h(p; \, \bR_i)+(p'-p) \cdot q_i(p).
\end{equation}
Finally, $q_i(p) \in \partial h(p; \, \bR_i)$ because the above holds for any $p' \in \RR^d$.

Next, because $h(p; \, \bR_i)$ is concave, its subgradient
$\partial h(p; \, \bR)$ is a measurable random set \citep[Remark 30]{shapiro2021lectures}.
Furthermore, because $\EE{h(p; \, \bR_i)} = H(p)$ by Lemma \ref{lemm:H},
our assumption that $H(\cdot)$ is differentiable at $p$ implies that $h(\cdot; \, R_i)$ is
almost surely differentiable at $p$ and $\EE{\nabla h(p; \, \bR_i)} = \nabla H(p)$
\citep[Theorem 7.46]{shapiro2021lectures}. Finally, because $q_i(p) \in \partial h(p; \, \bR_i)$
as shown above, almost sure differentiability of $h(p; \, \bR_i)$ implies that
$q_i(p) = \nabla h(p; \, R_i)$ almost surely, that $\EE{q_i(p)} = \nabla H(p)$.
Thus, by the strong law of large numbers,
\begin{equation}
\lim_{n \rightarrow \infty} \frac{1}{n} \sum_{i = 1}^n q_i(p) =_{a.s.} \nabla H(p),
\end{equation}
and \eqref{eq: mean-field-PM} follows from Assumption \ref{assum:sigma},
by invoking homogeneity and Lipschitz-continuity of $\rho(\cdot)$ (see Lemma \ref{lemma: rho-ball}
in the appendix).
\qed
\end{thm}

The above result gives the scaled costs for both the direct and price-based
mechanisms; it now remains to show that \eqref{eq: mean-field-DM} and the minimizer
of \eqref{eq: mean-field-PM} over $p$ match.
The function $\rho(\nabla H(p) - q_0)$ is not necessarily convex in $p$ because it depends on
$p$ through the differential of $H$; thus, classic duality theory for convex
functions does not apply here. The following result nonetheless establishes
optimality of the price-based mechanism under one further assumption:
We require sufficient stochasticity in the $\bR_i$ for the mean-field cost function $\E[h(p;\bR)]$
to be smooth and differentiable at its maximizer(s). We note that this phenomenon (i.e., differentiability at
the maximizer) is generally not true for the individual cost functions themselves.\footnote{We
emphasize that this randomness pertains to the distribution of consumption profiles $\bR_i$ rather
than consumer preferences. For example, in our Phoenix experiment discussed in Section \ref{sec: exp},
all consumers have similar preferences---yet heterogeneity in house characteristics (e.g., differences
in home insulation properties, cooling system performance, etc.) provide sufficient randomness
to justify Assumption \ref{assum: diff-basic}.}

\begin{assumption} \label{assum: diff-basic}
The cost $h(z, \, \bR_i)$ as in \eqref{eq:hp}
is almost surely differentiable at any solution to
\begin{equation}
\label{eq:zstar}
z^* \in \argmax\cb{H(z) - z \cdot q_0 : z \in \mathcal{P}}.
\end{equation}
\end{assumption}

\begin{thm}\label{thm: zero-gap-full}
Under the conditions of Theorem \ref{thm: zero-gap-basic},
suppose furthermore that Assumption \ref{assum: diff-basic} holds.
Then $H(z)$ is differentiable at $z^*$, and we have
\begin{equation}
\label{eq:zerogap}
\rho\p{\nabla H(z^*) - q_0} = H(z^*)-z^* \cdot q_0.
\end{equation}
Furthermore, using prices $p^* = \lambda \, z^*$ with any $\lambda > 0$
asymptotically solves the optimal demand response problem.
\proof
As noted above, \( H(\cdot) \) and \( h(\cdot, \bR) \) are both concave; thus, for any $z \in \RR^d$,
\( H(\cdot) \) is differentiable at \( z \) if and only if \( h(\cdot, \, \bR) \) is differentiable
at \( z \) almost surely \citep[Theorem 7.46]{shapiro2021lectures}.
Assumption \ref{assum: diff-basic} thus implies that $H(\cdot)$ is differentiable
for all $z^* \in \argmax\cb{H(z) - z \cdot q_0 : z \in \mathcal{P}}$.
Given that $z^*$ is a maximizer of $H(z)-z \cdot q_0$ over $\cP$, 
the first-order optimality condition then implies that
$$ \max\cb{ (z-z^*) \cdot \p{\nabla H(z^*)-q_0} : z \in \cP} = 0. $$
Furthermore, by Euler's homogeneous function theorem $z^* \cdot \nabla H(z^*)=H(z^*)$,
and so
$$ \max\cb{ z \cdot \p{\nabla H(z^*)-q_0} - \p{H(z^*) - z^* \cdot q_0} : z \in \cP} = 0. $$
Finally, given the definition of $\cP$ in \eqref{eq: cP}, this implies that
$$ \rho\p{\nabla H(z^*)-q_0} - \p{H(z^*) - z^* \cdot q_0} = 0. $$
By comparing this expression with \eqref{eq: mean-field-DM} and \eqref{eq: mean-field-PM}
we see that the asymptotic costs of the direct control method and pricing with $z^*$
match (recall that $z^*$ is a maximizer of $H(p) - p \cdot q_0$), and thus
pricing with $z^*$ achieves optimal demand response in the large-sample limit. Finally, as discussed following
Definition \ref{defi:HEMS_price}, users with a price-responsive HEMS are unresponsive
to scaling prices by a positive scalar. Thus, pricing with $p^* = \lambda \, z^*$
for any $\lambda > 0$ results in the same consumption profiles---and also solves
the optimal demand-response problem in large samples.
\endproof
\end{thm}

Theorem \ref{thm: zero-gap-full} above showed that, under considerable generality,
pricing with prices of the form $p^* = \lambda \, z^*$ with $\lambda > 0$ and $z^*$
as in \eqref{eq:zstar} solves the demand-response problem. This
abstract result highlights the promise of the pricing mechanism, but also raise
some immediate practical questions such as:
\begin{enumerate}
\item Are energy markets with dynamic prices $p^*$ stable?
\item How should we pick the scale parameter $\lambda > 0$? Price-responsive HEMS
as in Definition \ref{defi:HEMS_price} may be unresponsive to price scale; however,
consumers, grid operators and regulators will of course care about the overall level
of prices.
\item When can we guarantee that optimal prices are positive? Negative prices are
known to sometimes arise in practice when renewable production overwhelms demand;
however, we may believe that healthy demand-response mechanisms would avoid negative
prices.
\end{enumerate}
This section gives some (largely reassuring) answers to these questions.

We start with a positive answer to the first question, and verify that our proposed pricing method
results in stable consumer behavior. Qualitatively, the reason this holds is
our Assumption \ref{assum: diff-basic} that $h(p, \, \bR_i)$
is almost surely differentiable at $p^*$. This assumption reflects a requirement that consumers
are price-responsive at $p^*$, thus enabling us to use prices to shift demand in a predictable way.

\begin{thm}
\label{thm:stable}
Let $z^*$ be an optimal price under the conditions of Theorem \ref{thm: zero-gap-full}.
Then:
\begin{enumerate}
\item For each customer, with probability one, there exists a unique optimal energy consumption profile that optimizes the consumer objective under the price $z^*$, i.e., for consumer $i$ the following set is almost surely a singleton:
\begin{equation}
\Theta_i(z^*) = \cb{q\in \bR_i: q_i \cdot z^* = h(z^*; \, \bR_i)}.
\end{equation}
\item 
The average consumption profile of consumers under the optimal price $z^*$ is unique, i.e., $\E[\Theta_i(z^*)]$ is a singleton,
and furthermore

\begin{equation}
\E[\Theta_i(z^*)]=\Theta(z^*), \ \ \ \ \Theta(z^*) = \cb{ q \in \E[\bR], \, q \cdot z^* = h(z^*; \, \E[\bR])}.
\end{equation}
We denote this average consumption profile by $\bar{q}(z^*)$, i.e., $\Theta(z^*)=\{\bar{q}(z^*)\}$.
\end{enumerate}
\end{thm}

Next, given that our model is agnostic to the scale parameter $\lambda$,
this parameter can simply be chosen to match a target grid revenue.
Given Theorem \ref{thm:stable}, we know that average per-consumer grid revenue from net demand with prices
$p^*$ is stable, and equal to $U(p^*) = (\bar{q}(p^*)-q_0) \cdot p^*$. Thus, given a candidate price schedule $z^*$ (e.g., as in \eqref{eq:zstar}), and a target per-consumer revenue $C$,
and provided that $U(z^*) > 0$, we can achieve optimal demand-response at the
target per-user revenue, i.e., with $U(p^*) = C$, by using with prices\footnote{Under
the proposed prices, revenue is
$(q(p^*)-q_0) \cdot p^* = (q(z^*)-q_0) \cdot p^* =  (q(z^*)-q_0) \cdot z^* \, C / \p{(q(z^*)-q_0) \cdot z^*} = C$.}
\begin{equation}
\label{eq:price_scale}
p^* = C \, z^* \, / \, U(z^*).
\end{equation}
The following result establishes conditions under which $U(z^*) > 0$, and so this
scaling can be applied.

\begin{coro}
\label{coro:nonzero}
Under the conditions of Theorem \ref{thm: zero-gap-full}, suppose furthermore
that for non-negative grid cost function $\rho(.)$ we have $\rho(x) = 0$ if and only if $x = 0$, and that $q_0 \not\in \EE{\bR_i}$.
Then, for any $z^*$ as in \eqref{eq:zstar} we have that $U(z^*)=(\bar{q}(z^*)-q_0) \cdot z^* > 0$.
\end{coro}

Finally, we verify below that prices are in fact positive, as expected, whenever the following two conditions stated in the next assumption are satisfied. 

\begin{assumption}\label{assum: monotone}

    We assume that the average indifference set stochastically dominates available renewables, i.e., $q \geq q_0$ componentwise for all $q \in \EE{\bR_i}$. 
\end{assumption}

\begin{coro}
\label{coro:positive}
Under conditions of Theorem \ref{thm: zero-gap-full} and Assumption \ref{assum: monotone},
 Suppose furthermore that the grid cost function \( \rho(\cdot) \) satisfies a monotonicity property\footnote{This
 monotonicity property is satisfied, e.g., for grid cost functions given by the $\ell_p$ norm of net demand $\|q-q_0\|_{\ell_p}$, or the non-negative component of net demand $\|\max(q-q_0,0)\|_{\ell_p}$.}
 whereby for each pair $q ,\, r \in \RR^d$ if either  $0 \leq r_j \leq q_j$ or $q_j \leq r_j = 0$ for each $j = 1, \, \ldots, \, d$, then also $\rho(r) \leq \rho(q)$.
Then the optimal prices $p^*$ as described in
the statement of Theorem \ref{thm: zero-gap-full} satisfy $p^* \geq 0$.
\end{coro}

\section{A First-order Algorithm for Optimal Pricing}
\label{sec: opt-price}

Given the theoretical foundation for pricing mechanisms, we develop a practical algorithm for computing optimal dynamic prices. Our approach relies on first-order optimization methods and requires only minimal information from HEMS, making it both efficient and scalable.

In the previous section, we found that dynamic pricing is mean-field
optimal for demand response under our indifference-set model, and any
prices $p^*$ of the form
\begin{equation}
\label{eq:opt_price_summary}
p^* = \lambda \, z^*, \ \ \ \
z^* \in \arg\max\limits_{z\in \cP} \cb{H(z)-z \cdot q_0}, \ \ \ \
H(z) = \EE{h(z; \, \bR_i)},
\end{equation}
for any $\lambda > 0$, where $h(z; \, \bR_i)$ is the minimal spending for a
consumer with indifference set $\bR_i$ and $q_0$ is the scaled renewable production.
In order to make use of the result in practice, we need a way of learning prices from data.
One natural way of doing so is via empirical maximization with
\begin{equation}
\label{eq:opt_price_empirical}
\hat{p} = \lambda \, \hat{z}, \ \ \ \
\hat{z} \in \arg\max\limits_{z\in \cP} \cb{\widehat{H}(z)-z \cdot q_0}, \ \ \ \
\widehat{H}(z) = \frac{1}{n}\sum\limits_{i=1}^n h(z; \, \bR_i),
\end{equation}
where again $\lambda > 0$ is planner-chosen scale parameter. Below, we verify that
such empirical maximization achieves optimal demand response in large samples.

For any $z\in \cP$ at which $H(.)$ is differentiable, the mean-field suboptimality gap (i.e., then mean-field gap between deploying this price and the direct method grid risk value) following \eqref{eq: mean-field-DM} and $\bar{q}(z)$ as in Theorem \ref{thm:stable} can be characterized as
\begin{equation}\label{eq:Delta}
\Delta(z)=z\cdot \big(\bar{q}(z) -q_0\big)-\mathsf{R}^\myom\,.
\end{equation}
Recall that when \( H(\cdot) \) is differentiable at \( z \), the average consumption \( \bar{q}(z) \) is the unique solution to \( \min_{q \in \mathbb{E}[\bR]} z \cdot (q - q_0) \) (see proof of Theorem \ref{thm:stable}). With this notation, Theorem \ref{thm: zero-gap-full} implies that for \( z^* \) as defined in \eqref{eq:opt_price_summary}, we have \( \Delta(z^*) = 0 \).

\begin{thm}
\label{thm:empirical_max}
Under Assumptions \ref{assum: R-basic} and 
\ref{assum:sigma}, let $\hat{z}_n$ be any sequence of prices satisfying 
\eqref{eq:opt_price_empirical} for $n$ consumers. Then 
\begin{equation}
\lim\limits_{n\to \infty} \Delta(\hz_n) =_{a.s.} 0,
\end{equation}
i.e., pricing with $\hat{z}_n$ is asymptotically optimal.
\end{thm}

The next question is then how to efficiently maximize the function \smash{$\widehat{H}(z)$}.
At face value, the problem \eqref{eq:opt_price_empirical} looks like a
zeroth-order optimization problem: We seek to maximize a random function by taking draws
at different parameter values. In general, zeroth-order problems are solvable
but not particularly well-behaved. In the best case scenario where we can noiselessly
query the same unit multiple times, the error bound of zeroth-order optimization
are a factor $\sqrt{d}$ worse than what could typically be achieved via first-order
methods with access to gradients \citep{duchi2015optimal}.\footnote{These
methods proceed by effectively taking numerical gradients by evaluating
$h(z, \, \bR_i), \, h(z', \, \bR_i)$ for nearby values of $z, \, z'$.}
In addition, the situation gets much worse if there's any noise when the grid queries
each HEMS to get $h(z, \, \bR_i)$ (or if it's only possible to query each HEMS once);
in this case, zeroth-order optimization falls short of first-order behavior even
in terms of its rate of convergence \citep{shamir2013complexity}.

Luckily, however, the proof of Theorem \ref{thm: zero-gap-basic} reveals that
our setting affords us with natural estimates of the gradient of the objective
in \eqref{eq:opt_price_empirical}.
Since $\widehat{H}_n(\cdot)$ is a concave function, by an application of
Alexandrov's theorem we know that for almost every $z\in \cP$ we have
that $\widehat{H}_n(z)$ is differentiable at $z$. In addition, given the
formulation \eqref{eq:  smart-device}, it can be observed that $q_i(z)$
is a sub-gradient of $h(z; \, \bR_i)$ (see arguments before \eqref{eq: q_p_subgradient}
for more details). Putting these two results together, we realize that for
almost every $z \in \cP$, $\widehat{H}_n(.)$ is differentiable with derivative
\begin{equation}\label{eq: sample-avg-gradient}
\nabla \widehat{H}_n(z)=\frac{1}{n}\sum\limits_{i=1}^n q_i(z).
\end{equation}
In other words, we can evaluate the gradient of our objective simply by observing
the average consumption choice of each consumer at the current proposed prices.

The upshot is that we can compute system-level welfare gradients from data collected
in equilibrium and thus, as argued in \citet{wager2021experimenting}, can efficiently optimize the system using
familiar 1st-order optimization methods (all while inheriting their fast convergence properties).
Specifically, one can proceed via variants of the following simple meta-algorithm.
Given a current candidate price schedule $z_k \in \RR^d$:
\begin{itemize}
\item Send a (potentially scaled)\footnote{Prices could be scaled, e.g., using \eqref{eq:price_scale}.
Recall that, under our indifference-set model, consumer consumption is invariant to scaling prices
by a constant positive factor.}
price $p_k = \lambda \, z_k$ with $\lambda > 0$ to some set of HEMS $\mathcal{N}_k$.
\item Observe the average consumption $\bar{q}_k = \frac{1}{|\mathcal{N}_k|}\sum\limits_{i \in \mathcal{N}_k} q_i(p_k)$.
\item Form a gradient $g_k = \bar{q}_k - q_0$, and use it to update $z_k$ and obtain a new $z_{k+1} \in \cP$.
\end{itemize}
There is an extensive literature on first-order algorithms that can be used to implement this strategy \citep{bertsekas1999,parikh2014proximal}.
For example, one could use standard projected or constrained gradient descent algorithms that repeatedly
query all consumers (i.e., set $\mathcal{N}_k = \cb{1, \, \ldots, \, n}$), or
stochastic alternatives that learn by sequentially going through
the sample and only querying each customer once (i.e., with $\mathcal{N}_k = \cb{k}$).

\begin{rmk}
For a given price \(p\), the consumption profiles \(q_i(p)\) can be determined in two ways---resulting in two very different ways of running the above algorithm in practice. The first method is through actual experimentation, where users are exposed to the price \(p\) and their consumption profiles are observed after \(d\) periods. This approach can be time-consuming as it requires direct experimentation to identify the optimal price. The second method involves a communication system between the grid operator and smart management systems. In this approach, the grid operator briefly communicates a set of prices and queries users for their selected consumption profiles. The optimal price is then chosen based on the responses from users. This method can be significantly more efficient but requires more advanced Home Energy Management Systems (HEMS) that can communicate with the grid operator.
\end{rmk}

\section{Case Study: Pre-cooling in Phoenix, Arizona}\label{sec: exp}

In this section, we examine a sample grid of households equipped with Home Energy Management Systems (HEMS) on a summer day from Phoenix, Arizona, a particularly relevant testbed due to its extreme heat and abundant solar generation. To simplify the analysis, we only considered household electricity usage for cooling. In this setting,  homeowners input a desired range of indoor temperatures into their energy management systems, and then HEMS calculates a power trajectory based on the outside temperature, the house's thermal properties, user's preferences, and the electricity price.

For the outside temperature, we use the average hourly temperature data\footnote{\url{https://www.visualcrossing.com/weather/weather-data-services}} of Phoenix, Arizona in July 2023. We suppose that summer days in Phoenix are relatively consistent, using data from \( n = 31 \) days to get the sample mean \( \hat{\mu}_d^{\mathsf{temp}} \) and the sample covariance matrix \( \hat{\Sigma}_d^{\mathsf{temp}} \) for \( d = 24 \) hours. {In the next step, we follow the indoor temperature dynamics as given in Example \ref{ex: heat-dynamics}, i.e.,
\smash{$T^{\mathsf{in}}_{it} = T^{\mathsf{in}}_{i(t-1)} + \alpha_i (T^{\mathsf{out}}_{it} - T^{\mathsf{in}}_{i(t-1)}) + \beta_i q_{it}$},
where $\alpha_i \in \mathbb{R}$ characterizes the $i$-th house's thermal insulation, and $\beta_i < 0$ quantifies its cooling system's effectiveness.}
In our model, we assume that the outside temperature is the same for all households, and that each house has a fixed preference to maintain its inside temperature within the interval [20, 25] degrees, with an initial temperature of \(T^{\mathsf{in}}_0 = 24\) degrees. Following the temperature dynamics \eqref{eq: linear-heat}, we suppose heterogeneity across users for thermal parameters \( \alpha_i \) and \( \beta_i \), and 
let\footnote{These choices for $\alpha_i$ and $\beta_i$ are made to obtain proportionally meaningful energy levels for the consumption of HVAC systems over one hour. We also intentionally choose small $\alpha_i$ values (indicating good insulation properties) to better accommodate pre-cooling.}
\begin{align*}
\alpha_i&\overset{\mathsf{iid}}{\sim} \mathsf{Unif}(0.05,0.08)\,,\\
\beta_i &\overset{\mathsf{iid}}{\sim} \mathsf{Unif}(-0.35,-0.25)\,.
\end{align*}
We use Arizona's hourly renewable energy data\footnote{\url{https://www.eia.gov/opendata/browser/electricity/rto}} for solar in July 2023 to obtain the sample mean \( \hat{\mu}_d^{\mathsf{renewable}} \) and the sample covariance matrix \( \hat{\Sigma}_d^{\mathsf{renewable}} \) for available renewable.
We consider the following grid cost function parametrized by $s\ge 1$:
\begin{align}\label{eq: sigma-s}
\sigma_s(Q; Q_0) = \left\| \max(Q - Q_0, 0) \right\|_{\ell_s} \,.
\end{align}
This cost function only takes into account \textit{positive} net demand values, which require the operation of fuel-based power plants to meet the additional demand.

We first focus on the population-level (average values) quantities of outside temperature and the amount of available renewable to derive optimal dynamic price signals. Specifically, we consider the outside temperature $\hat{\mu}^{\mathsf{temp}}_d$ and the available renewable $\hat{\mu}_d^{\mathsf{renewable}}$ for $N=80$ users, and we compute the optimal price $p^*_s$ for various grid cost functions $\sigma_s$ using the Frank-Wolfe method for $40k$ iterations. To obtain prices that are periodic over 24 hours---ensuring that the price values for period 1 and period 24 are close in value---we consider a problem horizon of 72 hours. In this approach, the amounts of renewable and outside temperatures are replicated three times, creating a 72-hour period. We then solve the optimization problem over this 72-hour period and report the price values for the middle 24 hours as the final obtained prices. Figure \ref{fig:tmp-2} shows the optimal prices normalized by revenue for $ s = 1, 2, 4, $ and $ +\infty $. Here, we set the flat-rate pricing to be $p_0=\mathbf{1}_d$ (all ones), and then scale all $p^*_s$ values such that the revenue under each pricing mechanism equals that of the flat-rate pricing.

We next use the OpenDSSDirect\footnote{\url{https://pypi.org/project/opendssdirect.py/}} package to simulate the energy grid and solve the power flow equations, considering both active and reactive power. We suppose that HVAC system power factors are drawn i.i.d. from $\unif(0.8,0.95)$, having available active power for each household, this would allow us to compute the reactive power as well. Each household cooling system is connected to buses with a nominal voltage of $120$ volts. After solving the power flow equation for this grid, we computed the voltage variations at each period. Figure \ref{fig:voltages} shows the boxplots of the percentage variation in voltage (averaged across users) over $d$ periods under different pricing schemes. It is evident that the overall voltage variations of devices remain within a 0.36$\%$ deviation from the nominal value at the connected buses, demonstrating the system's stability under varying loads and dynamic pricing

For obtained prices $p^*_s$, we follow the HEMS response function \eqref{eq: smart-device} to obtain users consumption profiles. We first compute the total demand and inside temperatures in response to the price signal $p^*_s$ for $s=1,2,4,$ and $+\infty$. The results are shown in Figures \ref{fig: avg-consumption} and \ref{fig: avg-temp} respectively for net demand and inside temperatures. Under dynamic pricing, the average indoor temperatures exhibit a pre-cooling phase before reaching the high-demand periods, resulting in a different demand curve. {In addition, a detailed view of the temperature profiles of all consumers under different price signals is shown in Figure \ref{fig: temperature_dp_detailed}, where, in each sub-figure, thinner curves represent the indoor temperature dynamics of individual consumers throughout a full day. It can be observed that the optimal price for the grid risk function $\sigma_s(.)$ incentivizes a larger portion of consumers, and more strongly, to pre-cool their houses for larger values of $s$. The major pre-cooling occurs during midday hours (10–15), when abundant renewable energy is available, which helps mitigate the sharp ramp-up during the late evening around 20:00.} For $\sigma_1,\sigma_2,\sigma_4$, and $\sigma_{\infty}$ the total grid cost function under flat-rate pricing is respectively 2147.56, 619.29, 347.21, and 236.12, whereas when the optimal dynamic prices $p^*_s$ are deployed, the obtained total grid cost values are respectively 2041.93, 515.96, 250.88, and 153.52. This implies approximately 4.9\%, 16.7\%, 27.7\%, and 35.0\% cost savings, respectively.

\begin{figure}[p]
    \begin{subfigure}[t]{0.48\textwidth} 
        \centering
        \includegraphics[scale=0.2]{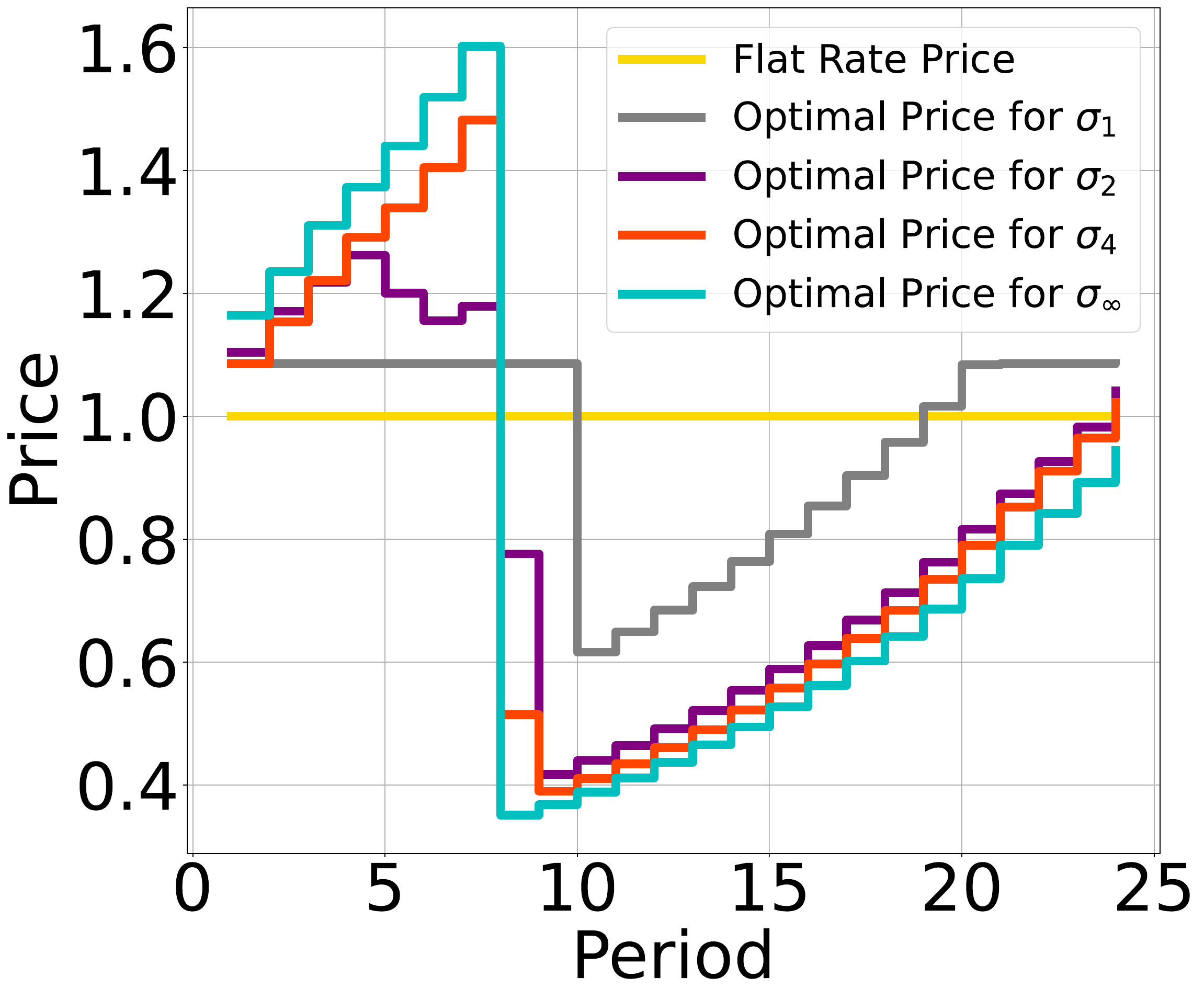}
        \caption{Values of optimal $p^*_s$ normalized by revenue}
        \label{fig:tmp-2}
    \end{subfigure}
    \hfill
    \begin{subfigure}[t]{0.48\textwidth} 
        \centering
        \includegraphics[scale=0.23]{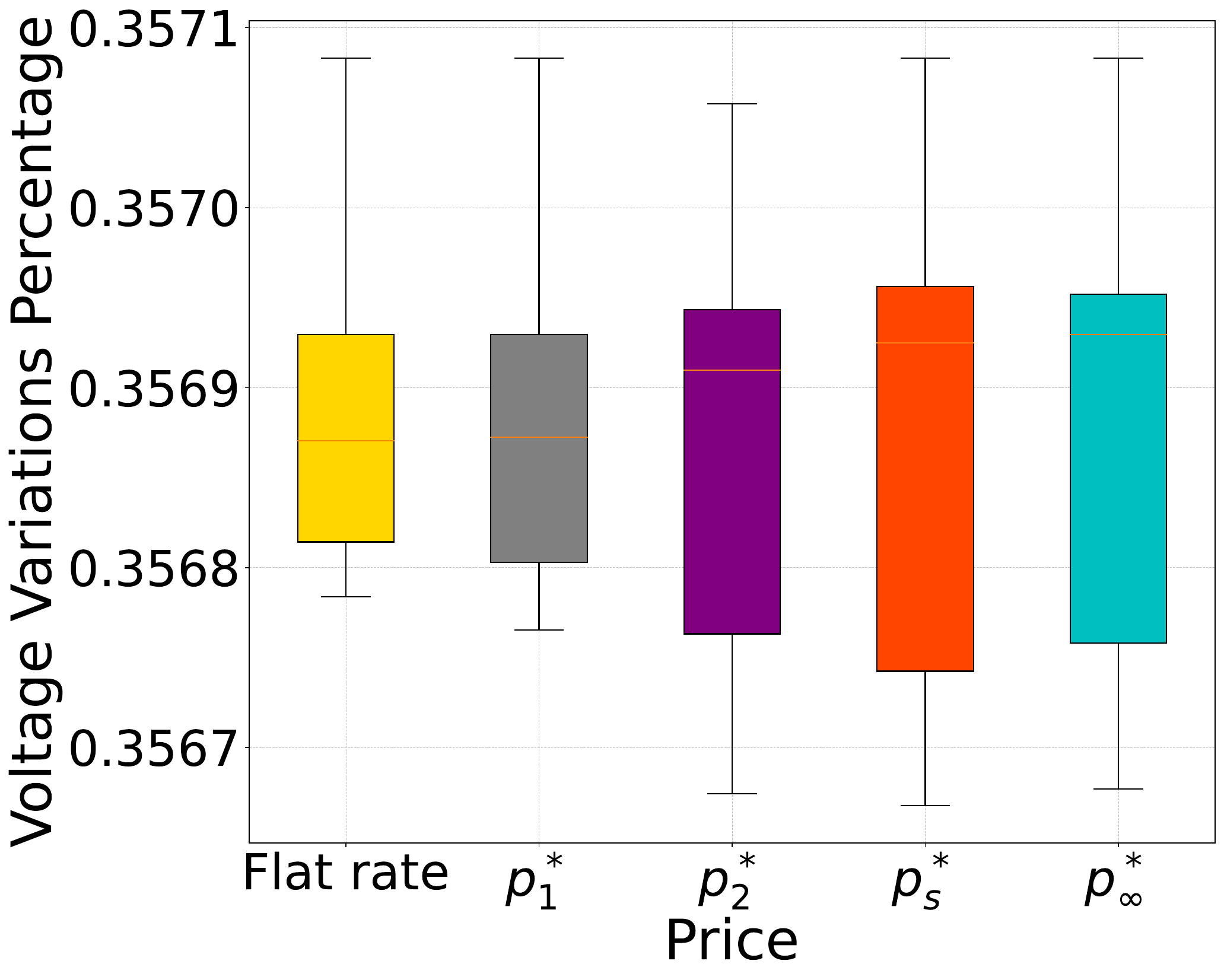} 
        \caption{Average voltage variations of different prices}
        \label{fig:voltages}
    \end{subfigure}
    \hfill

    \vspace{1cm} 
    
       \begin{subfigure}[t]{0.48\textwidth} 
        \centering
        \includegraphics[scale=0.22]{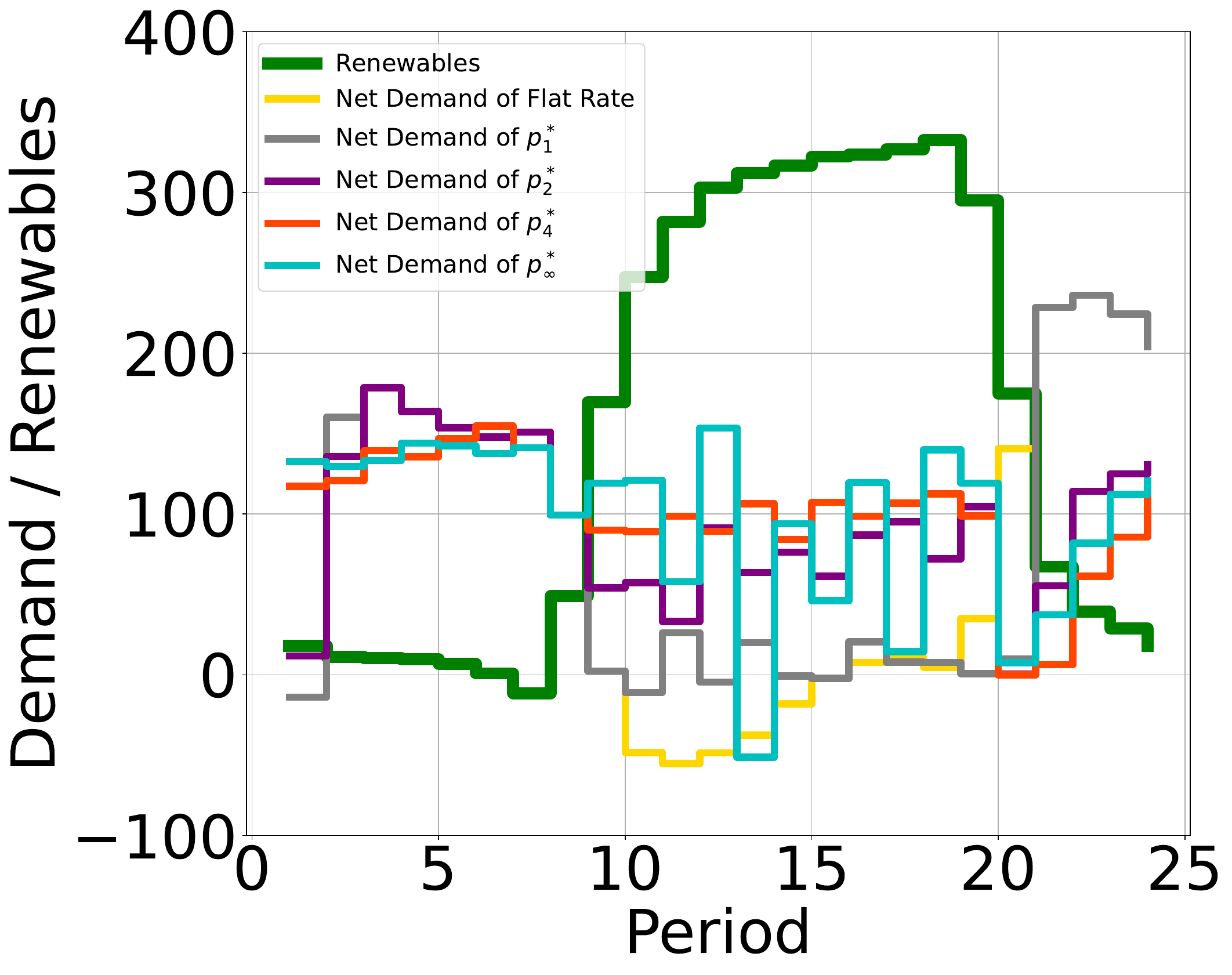}
        \caption{Demand curves of price $p^*_s$ and flat-rate pricing}
        \label{fig: avg-consumption}
    \end{subfigure}%
    \begin{subfigure}[t]{0.48\textwidth}
        \centering
        \includegraphics[scale=0.20]{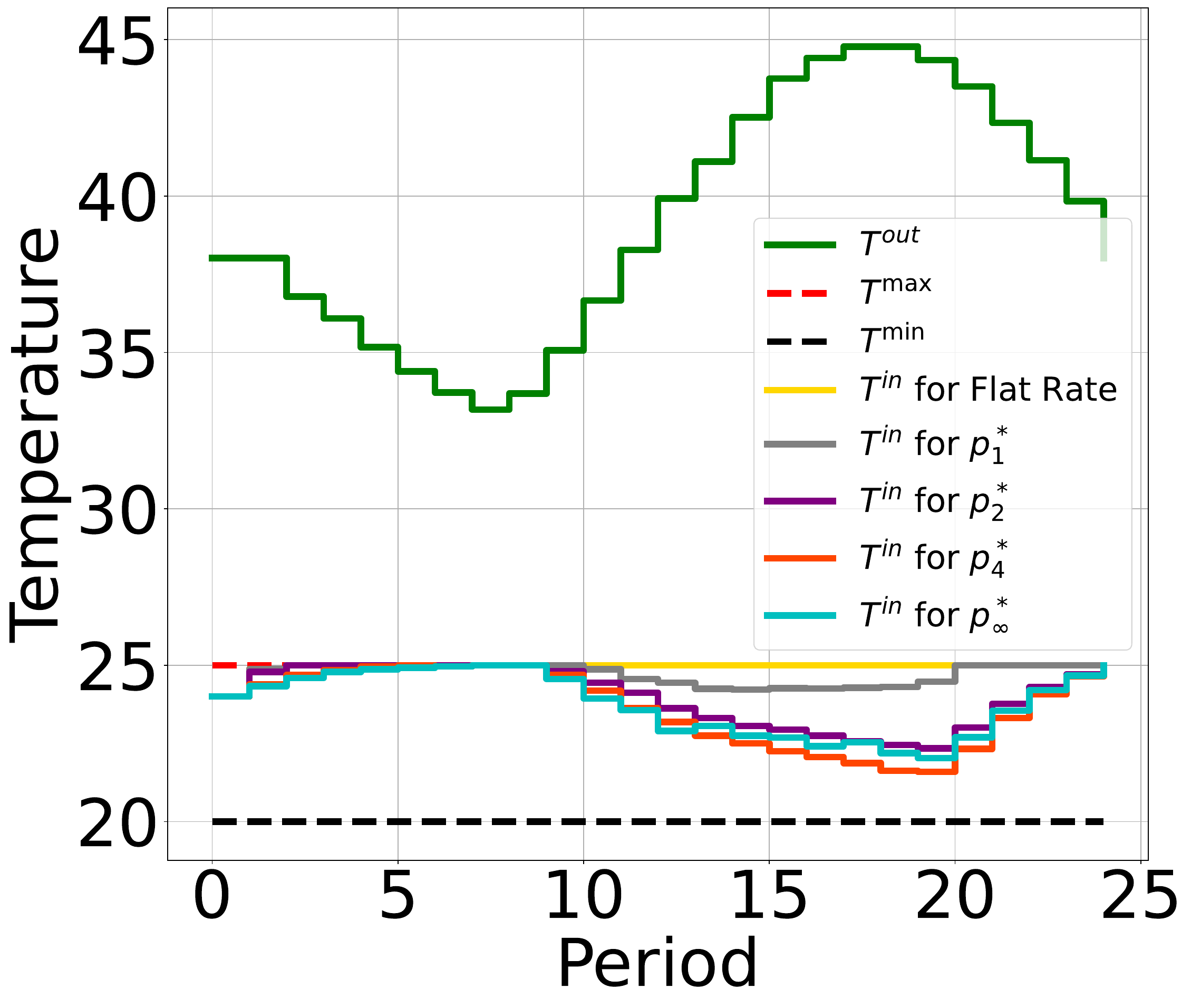}
        \caption{Inside average temperatures for $p^*_s$}
        \label{fig: avg-temp}
    \end{subfigure}
    \hfill 
    
    \caption{Results of numerical experiments evaluating the performance of dynamic pricing. Figure \ref{fig:tmp-2} shows the optimal price signal $p^*_s$ obtained for a variety of grid risk function $\sigma_s(.)$. Figure \ref{fig:voltages} shows voltage variations for $d$ periods across different pricing mechanisms. Figures \ref{fig: avg-consumption} and \ref{fig: avg-temp} display the demand curves and inside temperature profiles under $p^*_s$ and flat-rate pricing, respectively. Overall, deploying the dynamic pricing $p^*_s$ instead of flat-rate pricing results in cost savings of approximately 4.9\%, 16.7\%, 27.7\%, and 35.0\% for the grid cost functions $\sigma_1$, $\sigma_2$, $\sigma_4$, and $\sigma_\infty$, respectively.}
\end{figure}

\begin{figure}[p]
    \centering
     \begin{subfigure}[b]{0.19\textwidth}
        \centering
        \includegraphics[width=\textwidth]{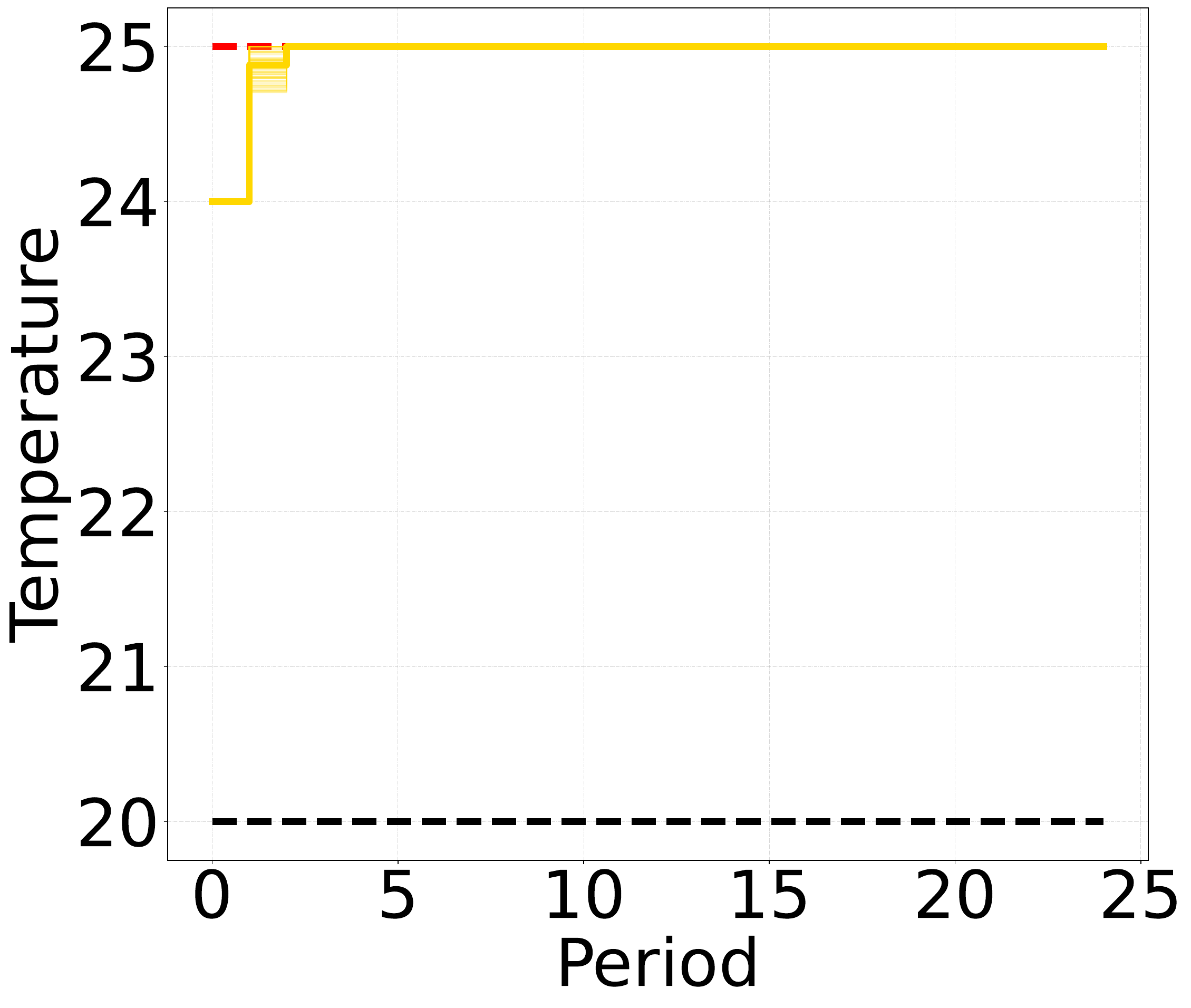}
        \caption{Flat Rate}
    \end{subfigure}
    \begin{subfigure}[b]{0.19\textwidth}
        \centering
        \includegraphics[width=\textwidth]{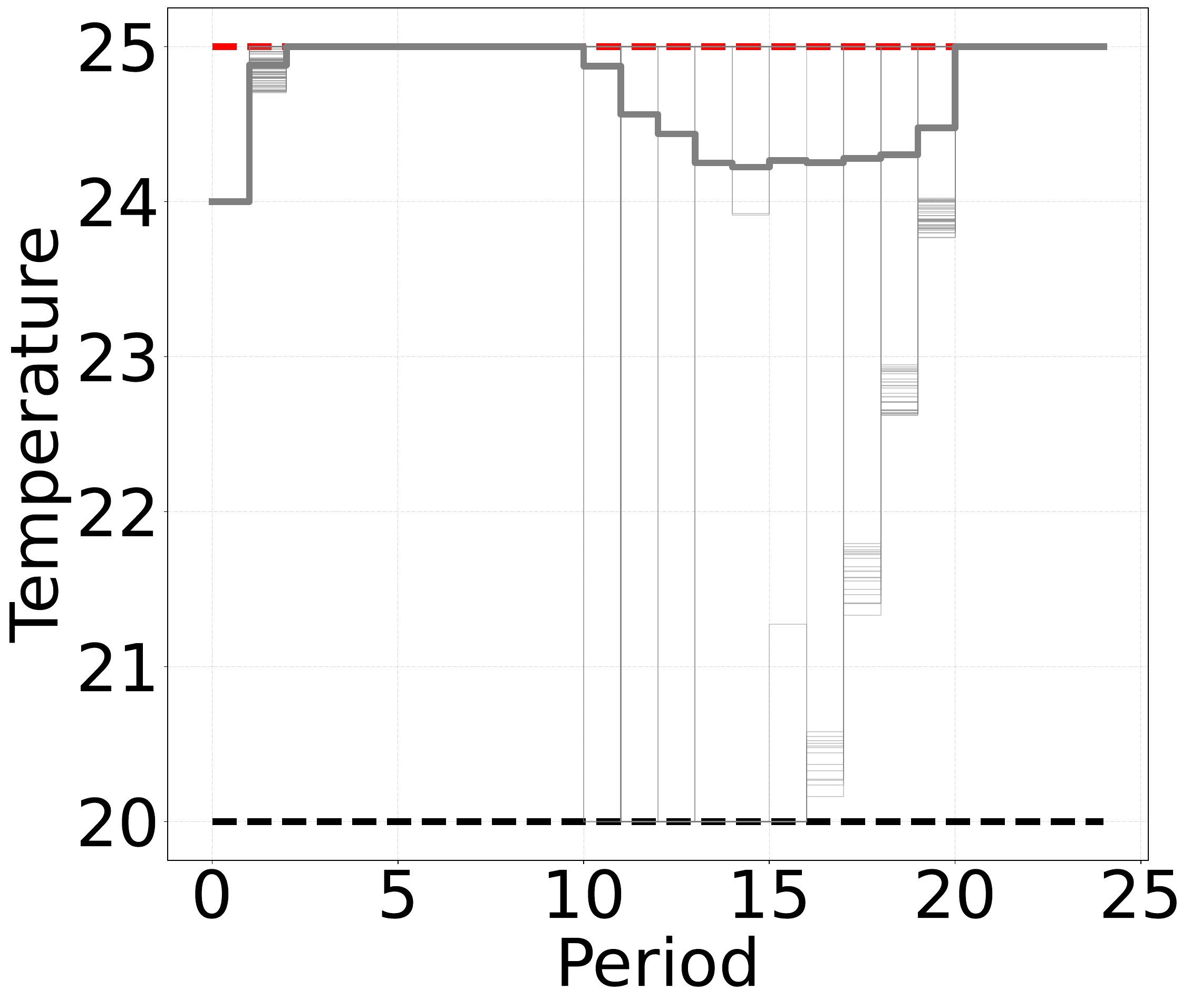}
        \caption{$p^*_1$}
    \end{subfigure}
    \begin{subfigure}[b]{0.19\textwidth}
        \centering
        \includegraphics[width=\textwidth]{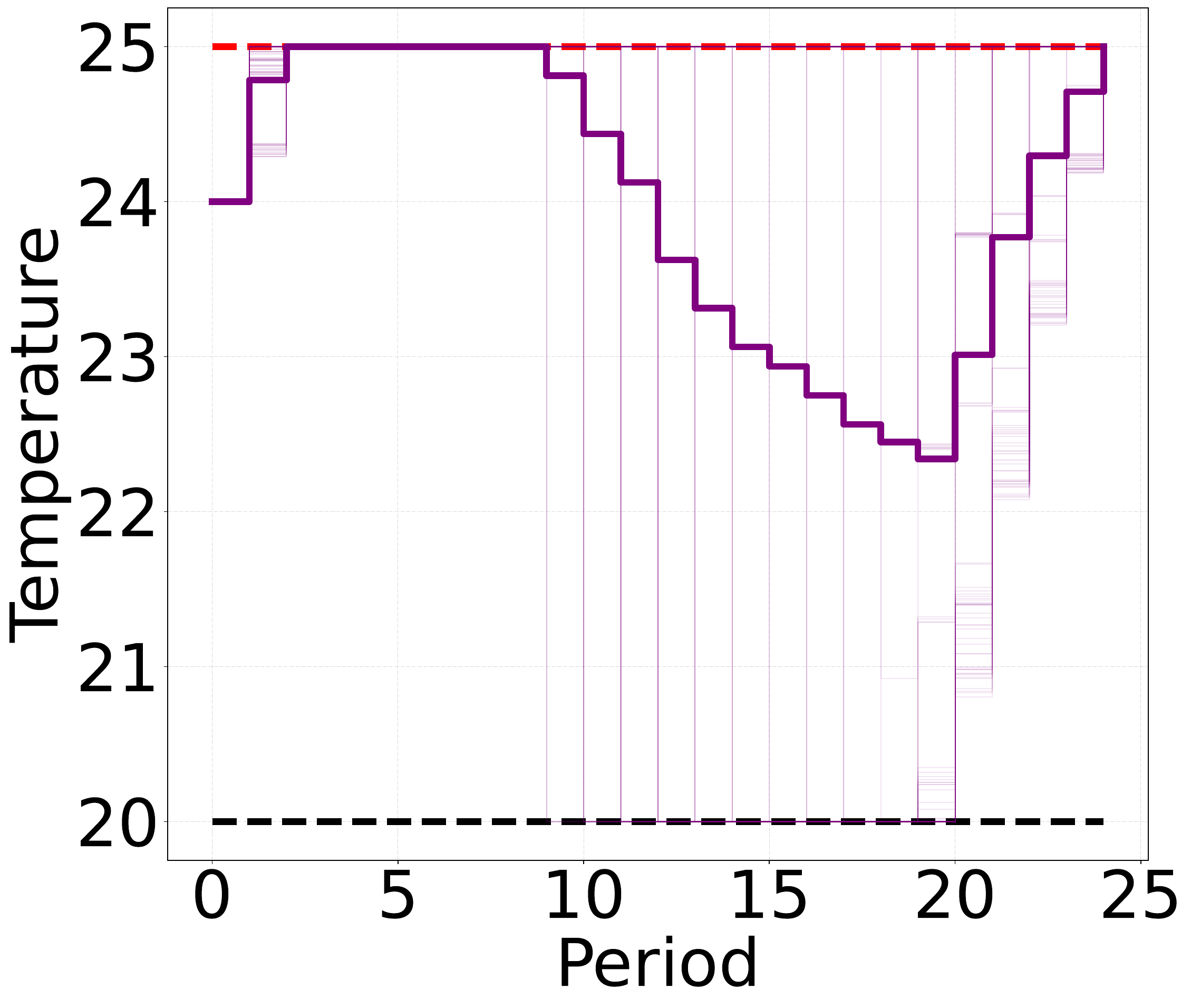}
        \caption{$p^*_2$}
    \end{subfigure}
    \begin{subfigure}[b]{0.19\textwidth}
        \centering
        \includegraphics[width=\textwidth]{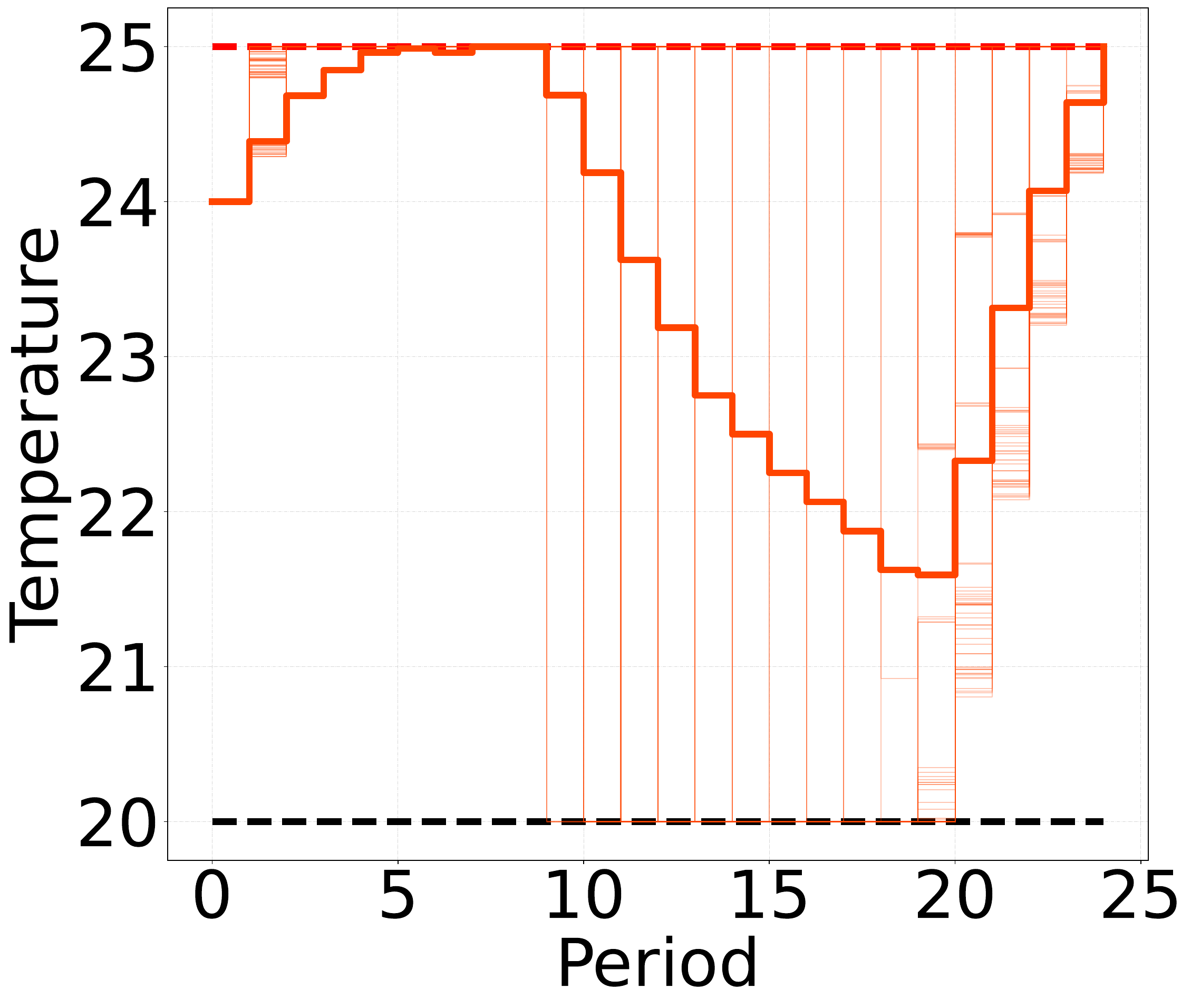}
        \caption{$p^*_4$}
    \end{subfigure}
    \begin{subfigure}[b]{0.19\textwidth}
        \centering
        \includegraphics[width=\textwidth]{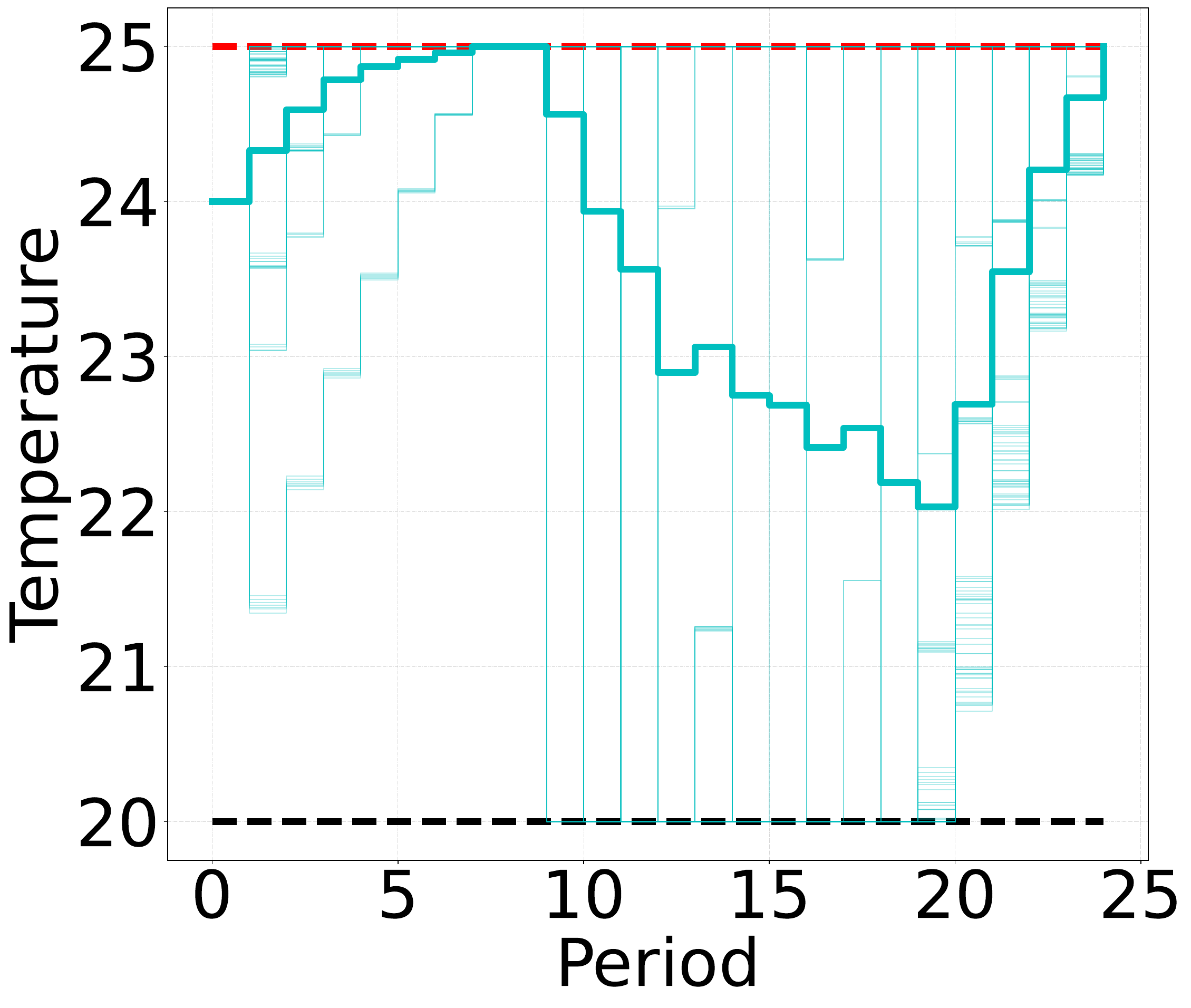}
        \caption{$p^*_\infty$}
    \end{subfigure}
    \caption{Comparison of all consumers' temperature profiles under different pricing strategies, $p^*_s$—the optimal price for the grid risk function $\sigma_s(.)$. This figure provides a detailed (zoomed-in) view of Figure \ref{fig: avg-temp}. Each thin curve represents the indoor temperature dynamics of an individual consumer.}
    \label{fig: temperature_dp_detailed}
\end{figure}

In the next experiment, we evaluate the performance of dynamic prices \( p^*_s \) when applied to different days. The objective is to examine how well the optimal price values, determined based on \textit{mean} outside temperatures and renewable energy availability, perform on subsequent days and their effectiveness in reducing specific grid cost functions. Notably, certain days may be colder or hotter, and we aim to assess the robustness of these pricing schemes. To this end, we consider 45 \textit{summer} days for the aforementioned 80 users, where, for each day, the outside temperature for each user is sampled from
\begin{equation}\label{eq:dist-out}
T^{\mathsf{out}}_i \overset{\mathsf{iid}}{\sim} \mathcal{N}(\hat{\mu}_d^{\mathsf{temp}}, \hat{\Sigma}_d^{\mathsf{temp}})\,.
\end{equation}
We keep the physical thermal properties of the houses, $\alpha$ and $\beta$, unchanged. Each day, we conduct four experiments, observing users' HEMS responses to four pricing signals: \( p^*_1 \), \( p^*_2 \), \( p^*_4 \), and \( p^*_\infty \). Our goal is to assess the effectiveness of \( p^*_s \) in reducing the grid cost function \( \sigma_s(\cdot) \) compared to flat-rate pricing. The exact percentage reductions in cost values compared to flat-rate pricing, calculated across all days and for various grid cost functions, are presented in Figure \ref{fig:heatmap}. In addition, a summary of these results is shown in Figure \ref{fig: violins}, where violin plots illustrate the outcomes for four values of \( s = 1, 2, 4, \) and \( +\infty \) over 45 days. It can be observed that even with fluctuations in outside temperature, the previously computed prices are effective in reducing the grid cost function. This effect becomes more pronounced as \( s \) increases, although with greater variation. We further demonstrate the total demand curves under \( p^*_s \) and flat-rate pricing for all 45 days in Figure \ref{fig: days-demand}. It is evident that the duck-shaped net-demand curve under flat-rate pricing is effectively flattened under dynamic pricing.

Finally, we compute the temperature profiles under \( p^*_s \) for various values of \( s \) and for flat-rate pricing across all 45 days, with the results displayed in Figures \ref{fig: days-temp}. To better illustrate the variation among different days, we separately plot the inside temperature dynamics in Figure \ref{fig: days-temp-zoomed-in}.
We see that our proposed algorithm in fact lowers indoor temperatures during the day when solar
power is available---as one would conceptually expect with pre-cooling.

\begin{figure}[p]
    \centering
    \begin{subfigure}[t]{0.98\textwidth}
        \centering
        \includegraphics[scale=0.1]{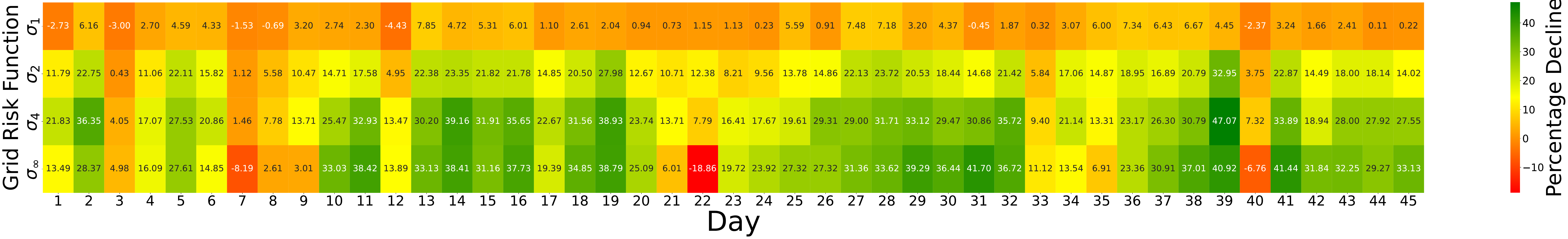}
        \caption{Daily percentage decline in grid cost function $\sigma_s$ when deploying $p^*_s$ compared to flat-rate pricing}
        \label{fig:heatmap}
    \end{subfigure}
    
    \vspace{0.5cm} 
    
    \begin{subfigure}[t]{0.48\textwidth}
        \centering
        \includegraphics[scale=0.21]{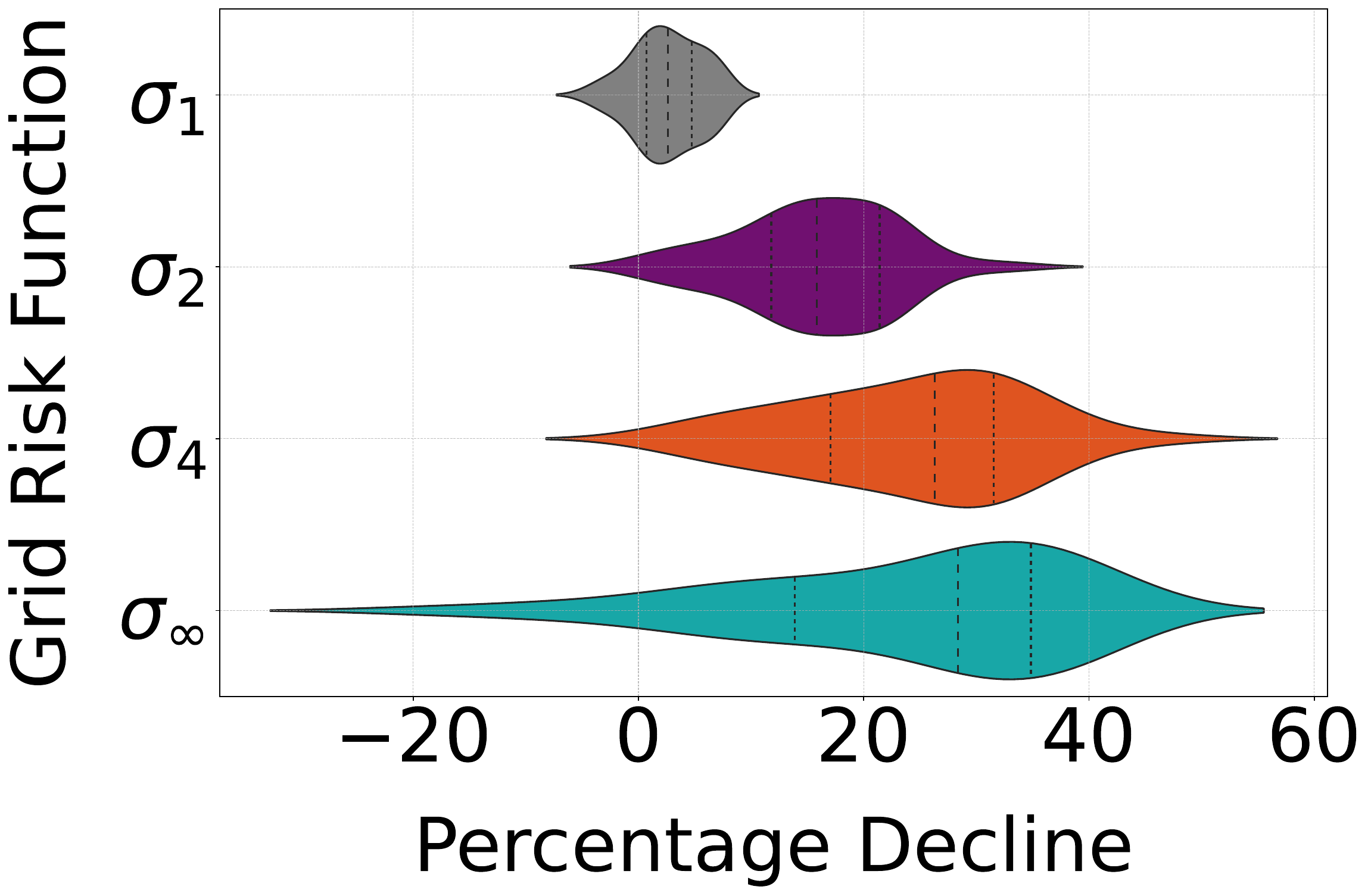}
        \caption{Violin plots of percentage decline in grid risk when using $p^*_s$ compared to flat-rate pricing}
        \label{fig: violins}
    \end{subfigure}
    \hfill 
    \begin{subfigure}[t]{0.48\textwidth}
        \centering
        \includegraphics[scale=0.15]{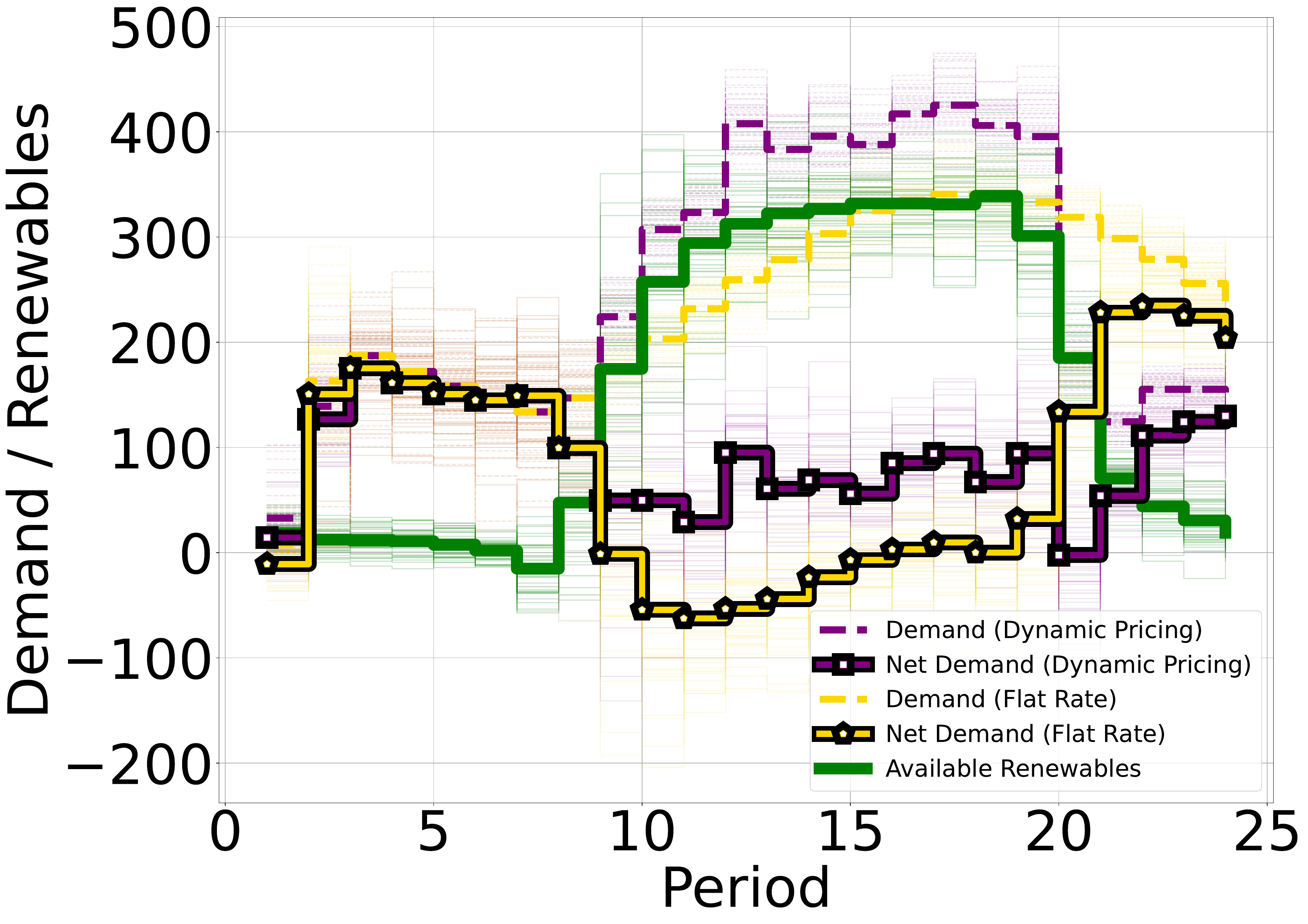}
        \caption{Aggregated user demand and available renewables under $p^*_2$ and flat-rate pricing for 45 days}
        \label{fig: days-demand}
    \end{subfigure}
    
    \caption{Comparison of grid cost, temperature dynamics, and demand profiles under optimal pricing $p^*_s$ and flat-rate pricing over 45 days.
    }
    \label{fig:days}
\end{figure}

\begin{figure}[t]
    \centering
    \begin{subfigure}[l]{0.45\textwidth}
        \centering
        \includegraphics[scale=0.22]{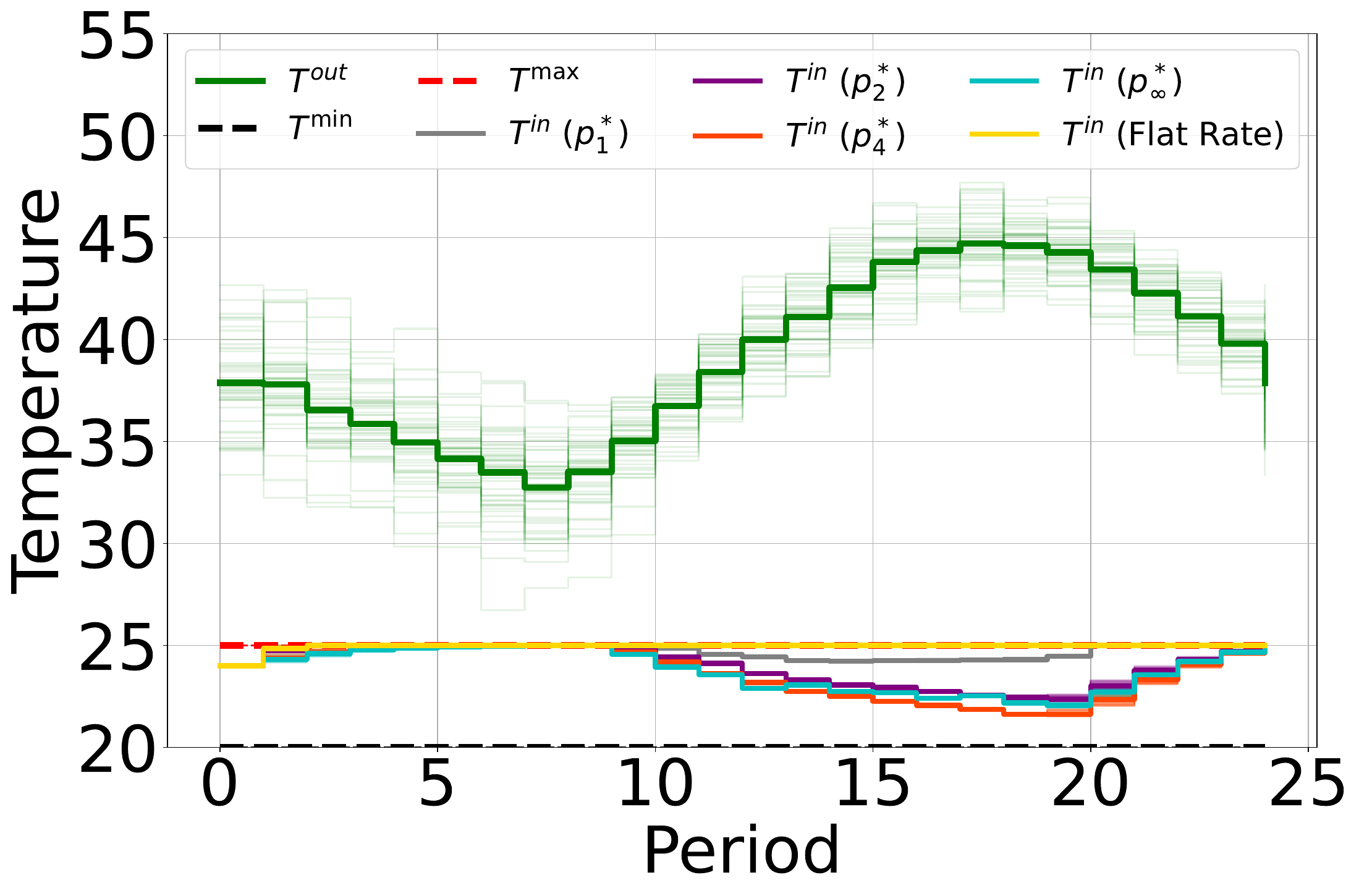}
 \caption{ Outside, min/max, and inside temperature for users under $p^*_s$ and flat-rate pricing over 45 days along with their average values}
        \label{fig: days-temp}
    \end{subfigure}%
    \hspace{0.05\textwidth} 
    \begin{subfigure}[r]{0.45\textwidth}
        \centering
        \includegraphics[scale=0.22]{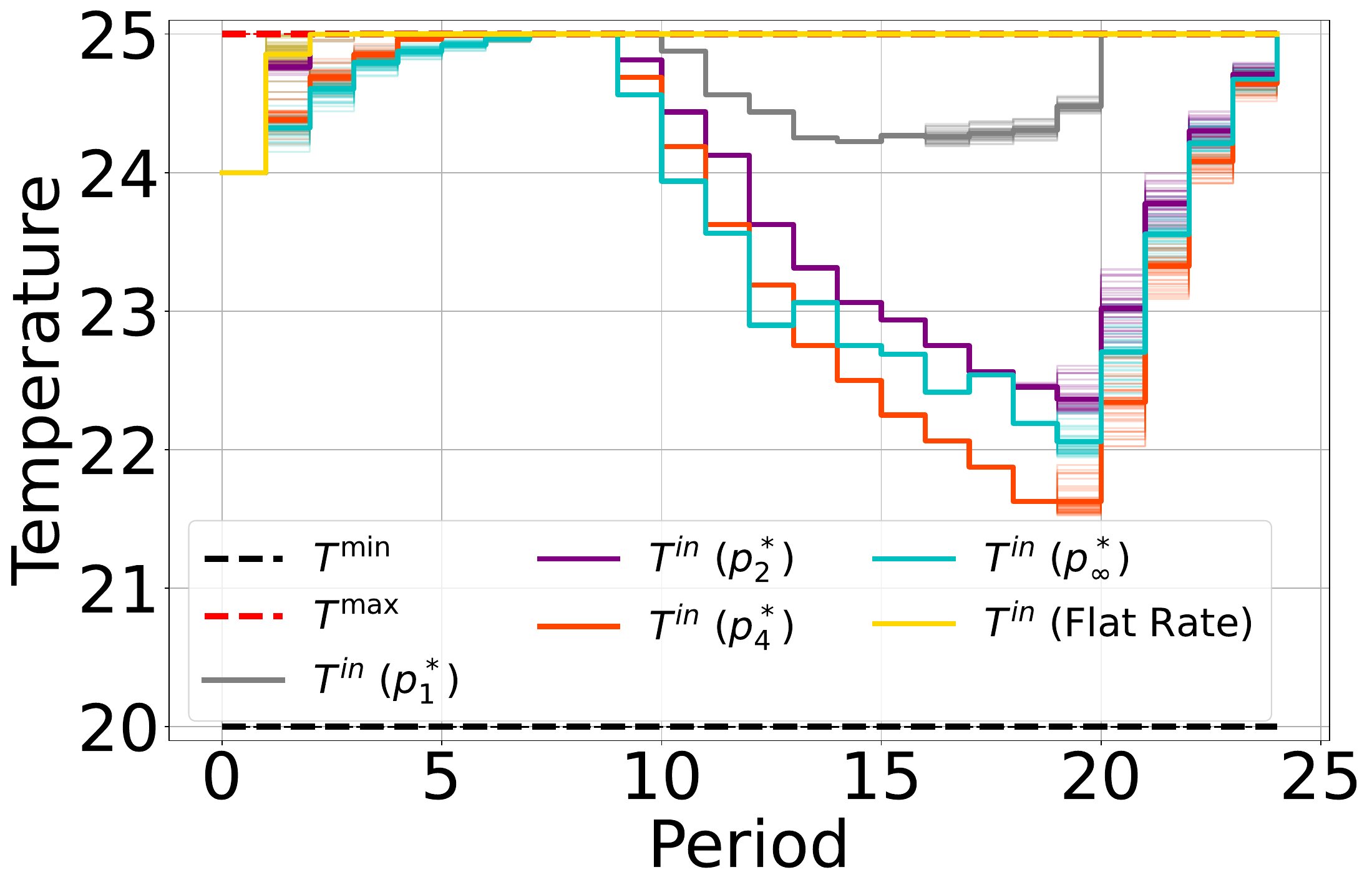}
        \caption{Detailed view of users' indoor temperature dynamics under $p^*_s$ and flat-rate pricing over a 45-day period, including their average values} \label{fig: days-temp-zoomed-in}
    \end{subfigure}
    \caption{Indoor temperature dynamics under dynamic pricing over a 45-day period. In this experiment, dynamic prices are computed based on the average outside temperature and available renewables, and tested over days with new realizations of outside temperature and renewable availability}
\end{figure}

\section{Discussion}
\label{sec:discussion}

In this paper, we demonstrated the potential of price-based demand response mechanism to
optimally leverage flexibility in consumer energy needs to achieve improved grid outcomes.
In a departure from existing work \citep[e.g.,][]{adelman2019dynamic,yang2012game,zhou2017incentive,zhou2019online},
we modeled user flexibility using indifference sets that simply formalize high-level consumer preferences---and
did not require consumers to reason about their full welfare function for energy use.
Our analytic results under this the indifference set model were enabled by fundamental results
on convex analysis for random sets \citep{aumann1965integrals,artstein1975strong}.

One interesting consequence of our indifference set model is that demand response becomes
essentially a pure substitution phenomenon: Under our model, HEMS respond to price signals
by shifting consumption across time periods rather than by curtailing overall consumption levels. We emphasize
that other models of consumer price response do not result in such behavior. For example,
\citet{allcott2011rethinking} reports experimental results in a setting where consumers
directly respond to time-varying energy prices (without using a HEMS or other technology
device as an intermediary), and finds that people do not shift consumption across time:
Instead, they simply reduced their consumption during peak hours. Whether or not consumers
respond to time-varying prices by shifting consumption across time plays a key role in
the overall effect of such programs \citep{jacobsen2020use,schittekatte2024electricity}.
In policy discussions, it is important to distinguish between demand-response programs
that assume consumers have the technological capacity to automatically respond to
prices---as is done here or in, e.g., \citet{adelman2019dynamic} and \citet{zhou2019online}---versus
programs where prices vary but consumers are simply expected to figure out how
to respond to them on their own \citep{allcott2011rethinking,cosmo2014estimating}.

There are a number of questions still left open in this paper. Here, we take the consumer
indifference sets, the total renewable production and the grid cost function as
given. However, if our proposed pricing mechanism were adopted, all these elements
may change endogenously. For example, if shifting energy use away from peak hours
becomes both practical and profitable for consumers, they may be willing to increase
the size of their indifference sets and thus provide further flexibility to the system;
and shifts in consumer demand patterns may open the door to increased renewable adoption.
It would be of considerable interest to study general-equilibrium implications of our
proposed demand-response mechanism. Meanwhile, on a practical level, it's clear that
optimal time-varying prices may be different across days, depending on a number
of factors such as weather, length of daylight and the day of the week. We are actively
working on extensions of our learning algorithm that can accommodate and respond to
such heterogeneity across days.

\bibliographystyle{plainnat}
\bibliography{mybib}

\newpage
\appendix

\section{Remaining Proofs}

\subsection{Proof of Proposition \ref{propo:warmup}}\label{proof: gap}
An indifference set $R$ with parameters $(a,b)$ can be reformulated as
\begin{equation}\label{eq: example-R}
R=\left\{(w a,(1-w) b) : w\in [0,1] \right\}\,.
\end{equation}
For a price signal $p=(p_1,p_2)$, each HEMS solves the following optimization
\[
w^*=\arg\min\limits_{w\in [0,1]}^{} \{ w a p_1 + (1-w) b p_2 \} \,.
\]
This gives the consumption profile $q(p)$ given in \eqref{eq:  smart-device} as 
\[
q(p)=(w^*a, (1-w^*)b)\,.
\]
In this case, the optimal solution $w^*$ is easy to characterize and is given by
\[
w^*=\begin{cases}
1\,,& ap_1< bp_2\,,\\
[0,1]\,, & a p_1=bp_2\,,\\
0\,,& ap_1>bp_2\,.
\end{cases}
\]
We suppose that users when they are indifferent, i.e., when $a p_1=bp_2$, they select period one ($w^*=1$). For an indifference set $R_i$ with parameters $(a_i,b_i)$ this yields
\[
q_i(p)=\Big( \ind(a_i p_1\le b_ip_2) a_i, \ind(a_i p_1> b_ip_2) b_i\Big)\,.
\]
Having $\bar{q}=\frac{1}{n}\sum\limits_{i=1}^n q_i(p)$, the average grid cost function for the price $p$ is $\|\bar{q}\|_{\infty}$, this allows us to characterize the grid risk under the price $(p_1,p_2)$ as the following:
\begin{align*}
\frac{1}{n}\hopt_n^{\mydp}(p)=\frac{1}{n} \max\bigg(\sum\limits_{i=1}^{n} \ind(a_i p_1\le b_ip_2) a_i, \sum\limits_{i=1}^{n} \ind(a_i p_1> b_ip_2) b_i\bigg)\,.
\end{align*}
The above formulation reveals that the grid cost function is scale-invariant with respect to price $p=(p_1,p_2)$, where for any $c>0$, using $cp$ yields the same grid cost value. In addition, if $p_1p_2\le 0$, then it is straightforward to observe that by having the non-positive element to zero, we obtain the same grid risk value. Put all together, to obtain the best price signal, we can focus on the following two classes of prices: i) all $p$ that $p_1+p_2=1$, for $p_1,p_2\ge 0$, and ii) all $p$ that $p_1+p_2=-1$, for $p_1,p_2\le 0$. The first class can be parametrized by a one-dimensional parameter $0\le z\le  1$ where $p_1=z$ and $p_2=(1-z)$, using this in the above yields
\[
U_n^+=\min\limits_{0\le z\le 1}^{}U_n^+(z)\,, \quad \frac{1}{n}U_n^+(z)=\frac{1}{n} \max\bigg(\sum\limits_{i=1}^{n} \ind\big(a_i z\le b_i(1-z)\big) a_i, \sum\limits_{i=1}^{n} \ind\big(a_i z> b_i(1-z)\big) b_i\bigg)\,.
\]
In addition, for the class of price signals $p_1+p_2=-1$ with $p_1,p_2\le 0$, we can consider $p_1=-z$, $p_2=-1+z$ for $z\in [0,1]$, using this in the formulation of $\hopt_n^{\mydp}(p)$ brings us
\[
U_n^-=\min\limits_{0\le z\le 1}^{}U_n^-(z)\,, \quad \frac{1}{n}U_n^-(z)=\frac{1}{n} \max\bigg(\sum\limits_{i=1}^{n} \ind\big(-a_i z\le b_i(-1+z)\big) a_i, \sum\limits_{i=1}^{n} \ind\big(-a_i z> b_i(-1+z)\big) b_i\bigg)
\]
Combining the above two classes of prices, then we can consider 
\[
\min\limits_{p}\hopt_n^{\mydp}(p)=\min\limits_{z\in [0,1]}U_n(z)\,, \quad U_n(z)=\min(U_n^+(z),U_n^-(z))\,.
\]
We then focus on computing the risk of the direct method. For this end, we recall the indifference set formulation \eqref{eq: example-R} and let the user $i$ be parameterized by $w_i\in [0,1]$, in this case we arrive at
\begin{align*}
\frac{1}{n}\hopt_n^{\myom}=\min\limits_{w_i\in [0,1]\,, \forall i\in [n]}^{}  \left\|\frac{1}{n}\sum\limits_{i=1}^{n}\begin{bmatrix} w_i a_i\\ (1-w_i) b_i \end{bmatrix} \right\|_{\ell_{\infty}}\,.
\end{align*}
We next use the dual formulation of $\ell_\infty$-norm to get 
\begin{align*}
\frac{1}{n}\hopt_n^{\myom}=\min\limits_{w_i\in [0,1]\,, \forall i\in [n]}^{}  \max\limits_{\|r\|_{\ell_1}\le 1}^{}\frac{1}{n}\sum\limits_{i=1}^{n}
\left(r_1 w_i a_i + r_2 (1-w_i) b_i  \right)\,.
\end{align*}
Given that the above objective function is bilinear in $r$ and $w_i$'s, we can change min/max with max/min and arrive at
\[
\frac{1}{n}\hopt_n^{\myom}=\max\limits_{\|r\|_{\ell_1}\le 1}^{}\min\limits_{w_i\in [0,1]\,, \forall i\in [n]}^{}\frac{1}{n}\sum\limits_{i=1}^{n}
\left(r_1 w_i a_i + r_2 (1-w_i) b_i  \right)\,.
\]
As the objective function is decoupled with respect to variables $w_i$, we can take the minimization inside and get
\[
\frac{1}{n}\hopt_n^{\myom}=\max\limits_{\|r\|_{\ell_1}\le 1}^{}\frac{1}{n}\sum\limits_{i=1}^{n}
\min\limits_{w_i\in [0,1]}^{}\left(r_1 w_i a_i + r_2 (1-w_i) b_i  \right)\,.
\]
The inside optimization is easy to characterize, where the optimal solution is given by
\[
w^*_i=\begin{cases}1, &r_1a_i\le r_2 b_i,\\
0, &r_1a_i> r_2 b_i\,.
  \end{cases}
\]
Using this yields
\[
\frac{1}{n}\hopt_n^{\myom}=\max\limits_{\|r\|_{\ell_1}\le 1}^{}\frac{1}{n}\sum\limits_{i=1}^{n}\left(r_1a_i\ind(r_1 a_i \le r_2 b_i)+r_2 b_i \ind(r_1 a_i > r_2 b_i) \right)\,.
\]
In the next step, we first note that given the non-negative values for $a_i,b_i$, it is easy to observe that $r_1,r_2$ to maximize the above objective must be non-negative, and the constraint can be written as $r_1+r_2=1$, for $r_1,r_2\ge0$. In addition,  it can be seen that the optimal $r^*$ must be on the boundary of unit $\ell_1$-ball, otherwise if $\|r^*\|_{\ell_1}<1$, we can easily consider $\tilde{r}=r^*/\|r^*\|_{\ell_1}$ and strictly increase the value of the objective function. This enables us to focus on all $r$ values that $r_1+r_2=1$,  by considering $z\in [0,1]$ and letting 
$r_1=z, r_2=(1-z)$, this brings us 
\begin{align*}
\frac{1}{n}\hopt_n^{\myom}&=\max\limits_{0\le z\le 1}^{} \frac{1}{n}\sum\limits_{i=1}^{n}\left( (1-z)b_i\ind\Big((1-z) b_i< z a_i \Big)+ z a_i \ind\Big((1-z)b_i \ge z a_i \Big)\right)\,.
\end{align*}
This completes the proof.

\subsection{Characterizing $\cP$}
\begin{lemma}\label{lemma: rho-ball}
For the grid const function $\rho(.)$ given in Assumption \ref{assum:sigma}, let  
\[
\cP\overset{\Delta}{=}\left\{p\in \reals^d: p^\sT r \le \rho(r)\,,\quad  \forall r\in \reals^d\right\}\,. 
\]
Then,  $\cP$ is a compact convex set, and its support function is $\rho(.)$. More precisely, we have
\[
\rho(q)=\max\limits_{z\in \cP}^{}z^\sT q\,,\quad \forall q\in \reals^d\,.
\]
In addition, $\rho$ is Lipschitz continuous with parameter $\|\cP\|$, where 
\[
|\rho(q_1)-\rho(q_2)|\le \|\cP\| \|q_1-q_2\|\,, \quad \forall q_1,q_2\in \reals^d\,.
\]
\end{lemma}

\begin{proof}
For the boundedness of $\cP$, note that if $\tp\in \cP$ by using $r=\tp$ we arrive at $ \|\tp\|_2^2 \le \rho(\tp)$. Using positive homogeneity of $\rho$ brings us
\[
\|\tp\|_2\le \frac{\rho(\tp)}{\|\tp\|_2}=\rho\left(\frac{\tp}{\|\tp\|_2}\right)\le \sup\limits_{x\in \mathbb{S}^{d-1}}^{}\rho(x)\,.
\] 
Given the continuity of $\rho$, on the unit sphere $\mathbb{S}^{d-1}$ we realize that $\rho$ is bounded. This implies that there exists a real-valued $M$ such that $\sup\limits_{x\in \mathbb{S}^{d-1}}^{}\rho(x)\le M$. Using this in the above, we get that for every $\tp\in \cP$ we have $\|\tp\|\le M$, this establishes the boundedness of $\cP$.  In addition, if $\{p_n\}_{n\ge 1}$ is a sequence of points converging to $p$ such that $p_n\in \cP$, it is straightforward to get that $p\in \cP$, showing the closeness of $\cP$. Put all together, it can be observed that $\cP$ is a compact set. 

To show convexity, note that if $p_1,p_2\in \cP$, then for every $\alpha\in (0,1)$, given that $p_1^\sT r\le \rho(r)$, and $p_2^\sT r \le \rho(r)$, then we have that $(\alpha p_1+(1-\alpha) p_2)^\sT r \le \rho(r)$, so $\alpha p_1+(1-\alpha) p_2  \in \cP$. 

We next show that  the support function of $\cP$ is $\rho(.)$. For this end, define the function $f(r)=\sup\limits_{p\in \cP}^{} (p^\sT r -\rho(r))$. For a fixed value of $r$, and a positive $\alpha$, we have the following
 \begin{align*}
f(\alpha r) &=\sup\limits_{p\in \cP}^{} (p^\sT (\alpha r) -\rho(\alpha r))\\ 
&=\alpha \sup\limits_{p\in \cP}^{} (p^\sT r - \rho(r))\\
&=\alpha f(r)\,.
 \end{align*}
 where in the penultimate relation we used the homogeneity of $\rho$. As this holds for every $\alpha\ge 0$ then we realize that $f(r)$ must be either $-\infty, 0,$ or $+\infty$. By definition of $\cP$, if $p\in \cP$, we know that $f(r)\le 0$, removing the $+\infty$ choice, implying that $f(r)$ can be $0$ or $-\infty$. We also have
 \begin{align*}
 \rho(r)&=\|r\|\rho\Big(\frac{r}{\|r\|}\Big)\\
 &\le \|r\| \sup\limits_{x\in \mathbb{S}^{d-1}}^{} \rho(x)\\
 &\le M\|r\|\,.
 \end{align*}
 In addition, as $|\sup\limits_{p\in \cP}^{} p^\sT r |\le \|r\|\|\cP\|\le M \|r\| $, we realize that $-2M\|r\|\le f(r)\le 0$, therefore $f(r)$ for every fixed $r$ can not be $-\infty$, therefore it is zero. This implies that $\sup\limits_{p\in \cP}^{} p^\sT r= \rho(r)$, meaning that $\rho(.)$ is the support function of $\cP$. For the Lipschitz property, for $i=1,2$ let
 \[
p^*_i=\arg\min\limits_{p\in \cP}^{}p^\sT q_i\,.
\]
Using this definition brings us 
\begin{align*}
\rho(q_1)-\rho(q_2)&=q_1^\sT p^*_1-q_2^\sT p^*_2\\
&=q_1^\sT (p^*_1-p^*_2)+ (q_1-q_2)^\sT p^*_2\\
&\le (q_1-q_2)^\sT p^*_2\\
&\le \|\cP\| \|q_1-q_2\|\,.
\end{align*}
By a similar chain of inequalities it can be shown that  $-\|\cP\| \|q_1-q_2\| \le \rho(q_1)-\rho(q_2)$, this completes the proof. 
\end{proof}

\subsection{Proof of Theorem \ref{thm:stable}}
We have
\begin{equation}
\Theta_i(z^*) = \cb{q \in \bR_i: q^\sT z^* = h(z^*;\bR)}.
\end{equation}
In the next step, by the Levin--Valadier theorem \citep[Theorem 7.22]{shapiro2021lectures}, given that
$h(.;\bR_i)$ is concave we have
\[
\partial_z h(z^*;\bR_i) = \conv\p{\Theta_i\p{z^*}}.
\]
In addition, under Assumption \ref{assum: diff-basic}, we have that $h(.;\bR_i)$ is differentiable at $z^*$ almost surely. By using this in the above, we realize that $\Theta_i(z^*)$ must be a singleton, given by $\Theta_i(z^*)=\{\nabla h(z^*; \bR_i)\}$.
We now move to prove the second part. From the definition of $H(.)$ we know that $H(z^*)=h(z^*;\E[\bR])$, therefore $\Theta(z^*)$ can be expressed as
\begin{equation}
\Theta(z^*) = \cb{q : q \in \EE{\bR_i}, \, q \cdot z^* = H(z^*)}.
\end{equation}
We once more apply the Levin–Valadier Theorem for the concave function $H(z)$ to arrive at
\[
\partial H(z^*) = \conv\p{\Theta\p{z^*}}.
\]
As $H(\cdot)$ is differentiable at $z^*$ as established in
Theorem \ref{thm: zero-gap-full}, $\Theta(z^*)$ must be a singleton, and $\Theta(z^*)=\{\nabla H(z^*)\}$, therefor $\bar{q}(z^*)=\nabla H(z^*)$. In addition, from Lemma \ref{lemm:H} we have $\E[h(z;\bR)]=h(z;\E[\bR])$, therefore $H(z)=\E[h(z;\bR)]$. We then use the concave property of $h(.;\bR)$, and $H(.)$ in conjunction with Theorem 7.46 of \citet{shapiro2021lectures} to get $\nabla H(z^*)=\E[\nabla h(z^*;\bR)]$. Combining this result with $\Theta_i(z^*)=\{\nabla h(z^*; \bR_i)\}$ implies that $\E[\Theta_i(z^*)]=\Theta(z^*)$.

\subsection{Proof of Corollary \ref{coro:nonzero}}
 From Theorem \ref{thm:stable} we have that $\bar{q}(z^*)$ is the unique solution for a consumer HEMS optimization of average indifference set $\E[\bR]$, formally 
 \[
 \bar{q}(z^*)=\arg\min\limits_{\E[\bR]} q^\sT z^*\,.
 \] 
 In addition, from the above optimization characterization  we know that $\bar{q}(z^*)\in \partial H(z^*)$ (see arguments used in \eqref{eq: q_p_subgradient}). Given the differentiablity of $H(.)$ at $z^*$ (see Assumption \ref{assum: diff-basic}), we obtain $\bar{q}(z^*)=\nabla H(z^*)$\,. We next write the first order optimality condition for the concave function $H(z)-z^\sT q_0$ over $\cP$, we arrive at the following
\[
(\nabla H(z^*)-q_0)^\sT(z-z^*)\le 0\,,\quad \forall z\in \cP\,.
\]  
This reads as
\[
 {z^*}^{\sT} (\nabla H(z^*)-q_0)\ge z^\sT (\nabla H(z^*)-q_0)\,, \quad  \forall z\in \cP\,.
\]
As this holds for all $z\in \cP$, we obtain
\[
{z^*}^{\sT} (\nabla H(z^*)-q_0)\ge \max \limits_{z\in \cP}^{} z^\sT (g^*-q_0)\,.
\]
In addition, the above maximization is achieved at $z^*\in \cP$, therefore it must be an equality, and can be rewritten as
\[
{z^*}^{\sT} (\nabla H(z^*)-q_0)=\max \limits_{z\in \cP}^{} z^\sT (\nabla H(z^*)-q_0)\,.
\]
In the next step, we deploy Lemma \ref{lemma: rho-ball} in the above to obtain
\begin{equation}\label{eq: g-star-equality}
{z^*}^{\sT} (\nabla H(z^*)-q_0)=\rho(\nabla H(z^*)-q_0)\,.
\end{equation}
We next use $\bar{q}(z^*)=\nabla H(z^*)$ in the above, and this amounts to $U(z^*)=\rho(\bar{q}(z^*)-q_0)$.
We next use $q_0\notin \E[\bR]$ to get that $\bar{q}(z^*)-q_0$ is not zero, accordingly $\rho(\bar{q}(z^*)-q_0)$ is not zero. Given that $\rho$ is non-negative valued, therefore $\rho(\bar{q}(z^*)-q_0)>0$. Plugging this into the above, yields $U(z^*)>0$, and completes the proof.

\subsection{Proof of Corollary \ref{coro:positive}}
For a vector $u\in \reals^d$, we adopt the following shorthand $u_+=\max(u,0)$, where the maximum is taken componentwise. We first claim that if $p\in \cP$, then $p_+$ also belongs to $\cP$. To show this, from Lemma \ref{lemma: rho-ball} we must show that 
\[
p_+^\sT r\le \rho (r)\,,\, \forall r\in \reals^d\,.
\]
We first note that 
\begin{align*}
p_+^\sT r&=\sum\limits_{\ell: p_\ell\ge 0} p^*_\ell r_\ell\\
&=\sum\limits_{\ell: p_\ell\ge 0, r_\ell \le 0} p^*_\ell r_\ell + \sum\limits_{\ell: p^*_\ell\ge 0, r_\ell\ge 0} p^*_\ell r_\ell\\
&\le \sum\limits_{\ell: p^*_\ell\ge 0, r_\ell\ge 0} p^*_\ell r_\ell\,.
\end{align*}
We then consider $r'$ where for $\ell\in [d]$, its $\ell$-th coordinate is defined as follows
\[
r'_\ell=\begin{cases}
r_\ell\,, & p_\ell \ge 0, r_\ell\ge 0\,,\\ 
0\,, & \text{o.w}\,.\\
\end{cases}
\]
From the above formulation, it is easy to observe that $r'^\sT p^*=\sum\limits_{\ell: p^*_\ell\ge 0, r_\ell\ge 0} p^*_\ell r_\ell$, therefore by using this in the above we get 
\[
p_+^\sT r\le r'^\sT p^*\,.
\]
In addition, it is easy to observe that $r'$ is such that either $0\le r'_\ell\le r_\ell$, or $r'_\ell=0$ when $r_\ell<0$. Therefore, it satisfies the monotonicity condition, and therefore $\rho(r')\le \rho(r)$. 
Put all together, we arrive at
\[
p_+^\sT r\le r'^\sT p^*\le \rho(r')\le \rho(r)\,,
\]
where $r'^\sT p^*\le \rho(r')$ follows from $p^*\in \cP$. 
As this holds for all $r$, therefore $p_+$ must belong to $\cP$. We now complete the argument and show that for $p^*$, if we consider $p=p^*_+$, then the objective function $H(p)-p^\sT q_0=\min\limits_{\E[\bR]}(q-q_0)^\sT p$ will not decrease, compared to $p=p^*$. This in conjunction with the fact that $p^*_+\in \cP$ implies that $p^*$ must have all non-negative coordinates, and completes the proof.
For this end, we consider $r=q-q_0$ for some $q\in \E[\bR]$, we then have
\begin{align*}
r^\sT p^*_+&=\sum\limits_{\ell: p^*_\ell\ge 0}^{} p^*_\ell r_\ell\\
&\ge \sum\limits_{\ell}^{} p^*_\ell r_\ell\,,
\end{align*}
where the last relation follows from the Assumption \ref{assum: monotone} that $q\ge q_0$ (componentwise), thus $r$ has non-negative coordinates. This implies that $r^\sT p^*_+\ge r^\sT p^*$, therefore 
\[
\min\limits_{\E[\bR]}^{} (q-q_0)^\sT p_+^* \ge \min\limits_{\E[\bR]}^{} (q-q_0)^\sT p^*\,.
\]
This completes the proof.


\subsection{Proof of Theorem \ref{thm:empirical_max}}
For random convex sets $\{\bC_i\}_{1\le i\le N}$ in $\mathfrak{C}_d$, from the alomst sure convergence result \eqref{eq: as-hasudorf} we have
\[
\lim\limits_{n\to \infty}^{}\sH\left(\frac{1}{n}\bigoplus\limits_{i=1}^{n}\bC_i ,\E[\bC]  \right)=0\,.
\]

In addition, for two convex sets $C_1, C_2 \in \mathfrak{C}_d$, the Hausdorff distance between $C_1,C_2$ can be written as the infinity norm of their support functions, formally we have
\[
\sH(C_1,C_2) = \sup\limits_{u \in \mathbb{S}^{d-1}}^{}\left|{\delta(u;C_1) - \delta(u;C_2)}\right|\,.
\]
Using this in the above, yields
\[
\lim\limits_{n\to \infty}^{}\sup\limits_{u\in \mathbb{S}^{d-1}}^{}\left|\delta\Big(u; \frac{1}{n}\bigoplus\limits_{i=1}^{N}\bC_i\Big) -\delta(u;\E[\bC])  \right|=0\,.
\]
We next use the linearity of support functions for Minkowski sum, where formally for convex sets $C_1,C_2$ we have 
\[
\delta(u;C_1\oplus C_2)=\delta(u;C_1)+\delta(u;C_2)\,,\quad  \forall u\in \mathbb{S}^{d-1}\,.
\]

This gives us
\begin{equation}
\lim\limits_{n\to \infty}^{}\sup\limits_{u \in \mathbb{S}^{d-1}}\bigg|\frac{1}{n}\sum\limits_{i=1}^{n}\delta\Big(u;\bC_i\Big) -\delta(u;\E[\bC])\bigg| = 0\,.\label{eq: a.s banach}
\end{equation}
We next use $h(z;\bR)=-\delta(-z;\bR)$, in conjunction with the definition of $H(z)=h(z;\E[\bR])$, and the fact that support function of a set and its convex hull are equal, i.e., $\delta(z;\bR)=\delta(z;\conv(\bR))$. This brings us
\begin{equation*}
\lim\limits_{n\to \infty}^{}\sup\limits_{u \in \mathbb{S}^{d-1}}\bigg|\frac{1}{n}\sum\limits_{i=1}^{n}h\Big(u;\bR_i\Big) -H(u)\bigg| = 0\,.\label{eq: a.s banach}
\end{equation*}
 We next deploy the positive homogeneity of function $h(.,\bR)$, and its mean-field value $H(.)$ to obtain 
\begin{align*}
\sup\limits_{z \in \cP}\bigg|\frac{1}{n}\sum\limits_{i=1}^{n}h\Big(z;\bR_i\Big) -H(z)\bigg|&=\sup\limits_{z \in \cP}\|z\|\bigg|\frac{1}{n}\sum\limits_{i=1}^{n}h\Big(\frac{z}{\|z\|};\bR_i\Big) -H\Big(\frac{z}{\|z\|}\Big)\bigg|\\
&\le \sup \limits_{z\in \cP} \|z\| \sup\limits_{u \in \mathbb{S}^{d-1}}\bigg|\frac{1}{n}\sum\limits_{i=1}^{N}h\Big(u;\bR_i\Big) -H(u)\bigg|\,.
\end{align*}
Using this in the above in along with the boundedness of $\|\cP\|$ (see Lemma \ref{lemma: rho-ball}) gives us

\begin{equation}
\lim\limits_{n\to \infty}^{}\sup\limits_{z \in \cP}\bigg|\widehat{H}_n(z) -H(z)\bigg| = 0\,.\label{eq: a.s banach}
\end{equation}
We then let $z^*\in \arg\max\limits_{\cP}^{} (H(z)-z^\sT q_0)$, it can be observed that $(H(z^*)-q_0^\sT z^*)=\mathsf{R}^\myom(q_0)$. In the next step, having established the uniform convergence of the sample average of users's cost function to their mean-field function $H(.)$, we write the following chain of relations with the new notation $q_{0,n}=Q_0/n$:
\begin{align*}
(H(z^*)-q_0^\sT z^*)-(H(\hz_n)-\hz_n^\sT q_0)&=
(H(z^*)-q_0^\sT z^*)-(\widehat{H}_n(z^*)- {q}_{0,n}^\sT z^*)\\
&\,~ + (\widehat{H}_n(z^*)- {q}_{0,n}^\sT z^*)- (\widehat{H}_n(\hz_n)- {q}_{0,n}^\sT \hz_n)\\
&\,~ + (\widehat{H}_n(\hz_n)- {q}_{0,n}^\sT \hz_n)-(H(\hz_n)-\hz_n^\sT q_0)\,.
\end{align*}
We next use the definition of $\hz_n$ as maximizer of $\widehat{H}_n(z)-z^\sT q_{0,n}$ to get
\begin{align*}
(H(z^*)-q_0^\sT z^*)-(H(\hz_n)-\hz_n^\sT q_0)&\le
(H(z^*)-q_0^\sT z^*)-(\widehat{H}_n(z^*)- {q}_{0,n}^\sT z^*)\\
&\,~ + (\widehat{H}_n(\hz_n)- {q}_{0,n}^\sT \hz_n)-(H(\hz_n)-\hz_n^\sT q_0)\,.
\end{align*}
Finally, we deploy the uniform convergence result of \eqref{eq: a.s banach} in along with almost sure convergence $Q_0/n\to q_0$ to realize that the above right-hand side in the mean-field is non-positive, thereby 
\[
\lim\limits_{n\to \infty}(H(\hz_n)-\hz_n^\sT q_0) \ge H(z^*)-q_0^\sT p^*\,.
\]
Given that $z^*$ is a maximizer of $H(z)-q_0^\sT z$ over $\cP$, combining this with the above we realize that with probability one the following holds
\begin{equation}\label{eq: a.s hp_n}
\lim\limits_{n\to \infty}(H(\hz_n)-q_0^\sT \hz_n) = H(z^*)-q_0^\sT z^*\,.
\end{equation}

\end{document}